\documentclass[11pt,oneside]{amsart}

\usepackage{url}
\usepackage[english]{babel}
\usepackage{amsthm, amsmath, amssymb, amsfonts, amscd}
\usepackage[svgnames]{xcolor} 
\usepackage[colorlinks,citecolor=DarkViolet,pagebackref=true,urlcolor=DarkViolet, bookmarksdepth=3, linkcolor=DarkCyan]{hyperref} 
\usepackage{kpfonts} 

\usepackage{stmaryrd}
\usepackage{listings}
\usepackage{graphicx}
\usepackage{tikz-cd}
\usepackage{indentfirst}
\usepackage{mathrsfs}
\usepackage{accents}
\usepackage{xcolor}
\usepackage{forest}

\usepackage{listings}
\usepackage{color}

\definecolor{darkspringgreen}{rgb}{0.09, 0.45, 0.27}
\definecolor{emerald}{rgb}{0.31, 0.78, 0.47}
\definecolor{indiagreen}{rgb}{0.07, 0.53, 0.03}
\definecolor{bulgarianrose}{rgb}{0.28, 0.02, 0.03}
\definecolor{brandeisblue}{rgb}{0.0, 0.44, 1.0}
\definecolor{bananayellow}{rgb}{1.0, 0.88, 0.21}

\def\ker{\operatorname{ker}}
\def\dim{\operatorname{dim}}

\def\exp{\operatorname{exp}}

\def\spec{\operatorname{spec}}
\def\supp{\operatorname{supp}}

\def\cone{\operatorname{cone}}

\def\grad{\operatorname{grad}}
\def\tube{\operatorname{tube}}

\def\Cayley{\operatorname{Cayley}}

\def\conv{\operatorname{conv}}
\def\p{\partial}

\def\th{\tilde{h}}

\def\Z{\mathbb{Z}}

\def\R{\mathbb{R}}
\def\C{\mathbb{C}}

\def\A{\mathbb{A}}
\def\F{\mathcal{F}}

\def\P{\mathbb{P}}
\def\O{\mathcal{O}}

\def\T{\mathbb{T}}
\def\X{\mathscr{X}}
\def\Y{\mathscr{Y}}

\newtheorem{theorem}{Theorem}[section]
\newtheorem{corollary}[theorem]{Corollary}
\newtheorem{lemma}[theorem]{Lemma}
\newtheorem{proposition}[theorem]{Proposition}

\newtheorem{definition-proposition}{Definition-Proposition}[section]
\newtheorem{definition-lemma}[theorem]{Definition-Lemma}

\address{Mykola Pochekai \\ Center for Geometry and Physics \\ Institute for Basic Science}
\email{mykola.pochekai@gmail.com}

\title{Geometry of the mirror models dual to the complete intersection of two cubics}
\author{Mykola Pochekai}
\begin{document}
\begin{abstract}
    We construct a natural crepant resolution of the Batyrev--Borisov mirror dual family to the complete intersection of two cubic hypersurfaces in $\mathbb P^5$. It is, similarly to the mirrors of quintic threefolds, a family over $\mathbb{P}^1$ with singular fibers over the set $\{0, \infty\} \cup \mu_6$. We compute an explicit height function producing the MPCP desingularization of the $B$-model toric ambient space. We compute the limiting mixed Hodge structures of the singular fibers. We show that the singular fiber over $\infty$ has maximal unipotent monodromy, whereas the singular fiber over $0$ is of a new type compared to the quintic case.
\end{abstract}
\maketitle
\thispagestyle{empty}

\tableofcontents

\section{Introduction}

One approach to mirror symmetry, originating in BCOV theory \cite{BCOV}, studies genus-one mirror symmetry and has been developed in a series of papers \cite{EFiMM,EFiMM2,EFiMM3}. In particular, genus-one mirror symmetry has been established for Calabi–Yau hypersurfaces in projective space. The next nontrivial case is complete intersections of two cubic hypersurfaces in $\P^5$. The present paper provides the geometric input needed to extend this program to the complete intersection of two cubic hypersurfaces in $\P^5$.

More precisely, we first construct a natural maximal projective star triangulation of the mirror polytope, which yields an explicit MPCP desingularization of the Batyrev--Borisov mirror family. This explicit model allows us to determine the geometry of all singular fibers of the compactified family and to compute the associated limiting mixed Hodge structures. Together, these results provide the detailed geometric input required by the genus-one mirror symmetry program.

The Batyrev--Borisov mirror construction produces a singular mirror family. Constructing an explicit MPCP desingularization therefore reduces to finding a suitable maximal regular star triangulation of the mirror polytope. We obtain a natural triangulation by relating the mirror polytope to the type $A_6$ alcoved triangulation of the corresponding Cayley polytope. See Theorem \ref{th:triangulation}.

\begin{theorem}
The type $A_6$ alcoved triangulation of the corresponding Cayley polytope
descends to a regular maximal
triangulation of the boundary of the mirror polytope. Coning from the origin yields a maximal regular star triangulation of the mirror polytope.
\end{theorem}

The resulting triangulation is a natural extension of the so-called edgewise subdivisions, constructed in \cite[Chapter III, Example 2.3]{KunMum}, which play a role in the theory of semistable reductions and were later studied independently in \cite[Section 2.3]{Esd}. The constructed triangulation is also the projection of a subcomplex of an alcoved triangulation \cite[Section 3]{LamPostnikovI}. 

Using this triangulation, we obtain an explicit MPCP desingularization of $\Y_\psi$ and determine the geometry of its singular fibers. The main geometric properties are summarized in the following theorem. We summarize the computation in the following theorem (Proposition \ref{th:infinity-is-mum}, Corollary \ref{cor:unipotency-of-0}, Theorem \ref{th:geometry-summary}, Theorem \ref{RhombusThoerem}). 

\begin{theorem}
Let $f\colon\Y\to\mathbb P^1$ be the compactified mirror family.
\begin{enumerate}
\item The fiber over each $\psi\in\mu_6$ has a unique ordinary double point;
\item The fiber over $0$ is the union of three irreducible components, the point $\psi = 0$ is a $K$-point;
\item The fiber over $\infty$ is the unique singular fiber of the total space, and its monodromy is maximally unipotent.
\end{enumerate}
\end{theorem}

The appearance of the K-point is the principal new feature distinguishing the present family from the hypersurface case studied in \cite{EFiMM3}, and plays an essential role in the analysis of the associated limiting mixed Hodge structures.


\section*{Acknowledgments}
The author would like to thank Dennis Eriksson for introducing him to the problem, for many valuable discussions, and for reviewing the text. He is also grateful to Marcus Berg, Michele Rossi, and Omid Amini for their helpful discussions and suggestions for improving the paper. This constitutes a part of the author's PhD thesis.

\section{Description of the dual nef-partition for intersection of two cubics}
\label{sec:dual-nef}
All schemes are considered to be of finite type over $\C$. By a variety, we mean a reduced, not necessarily irreducible, scheme of finite type over $\C$.

Consider a lattice $N$, which is a free abelian group of finite rank. Its dual lattice is denoted by $M := N^{\vee}$. A lattice polytope is defined as the convex hull of a finite set of points from $N$, and is a subset of $N_{\R} := N \otimes_{\Z} \R$. The notation $\partial P$ denotes the boundary of a polytope $P$.

If a full-dimensional polytope $P \subset N_{\R}$ contains $0_N$ in its interior, the dual polytope $P^{\vee} \subset M_{\R}$ can be defined in the standard way \cite[p.65]{CLS}. By $(n,m)$ where $n \in N_{\R}$ and $m \in M_{\R}$ we mean the standard perfect pairing between the $\R$-vector spaces $N_{\R}$ and its dual $M_{\R}$. Even if $P$ is a lattice polytope, the dual polytope $P^{\vee} \subset M_{\R}$ is not necessarily a lattice polytope. A full-dimensional lattice polytope containing $0$ in its interior is called reflexive if the dual polytope is also a lattice polytope. Reflexive polytopes will play a key role in later constructions since they correspond to Gorenstein Fano toric varieties \cite[Theorem 8.3.4]{CLS}. All polytopes considered in this work will be lattice polytopes; hence, we will simply use the term "polytope" to refer to a lattice polytope. Throughout this paper, we will refer to the Cox quotient construction, which is summarized in \cite[Theorem 5.1.11]{CLS}. 

Let's introduce the following lattices:
$$N=\{(x_1,...,x_6) \in \Z^6 : x_1+...+x_6=0\}, \qquad M = N^{\vee} = \frac{\Z^6}{\Z\{e_1+...+e_6\}}.$$
We write $e_i$ for the basis vector in $\Z^6$ and also use $e_i$ to denote the quotient class in $M$. Sometimes, when we want to emphasize that $M$ is the character lattice of the torus $\T_N$, we use a multiplicative basis $t_1,...,t_6$ with the condition $t_1 ... t_6 = 1$. In this case, we can write:
\begin{equation}\label{eq:m-ti-coordinates}
    M = \{t_1^{k_1}...t_6^{k_6} : (k_1,...,k_6) \in \Z^6\}, \quad t_1 ... t_6 = 1.
\end{equation}

We first remind the reader of the projective toric model of $\P^5$ used below. The polytope
$$\Delta = \conv (6e_1 - \sum_{i=1}^6 e_i, 6e_2 - \sum_{i=1}^6 e_i, ..., 6 e_6 - \sum_{i=1}^6 e_i) \subset N_{\R}.$$
satisfies $\P_{\Delta} \cong \P^5$. Here $\P_{\Delta}$ denotes the toric variety associated with the polytope $\Delta$; see \cite[Definition 2.3.14]{CLS}.

The normal fan $\Sigma_{\Delta}$ of the polytope $\Delta$ can be computed:
$$\Sigma_{\Delta}(5) = \{\sigma_1,...,\sigma_6\}, \qquad \sigma_i = \cone \{e_1,e_2,...,\widehat{e_i},...,e_6\} \subset M.$$
If $\Sigma$ is a fan, we will use $\Sigma(k)$ for the family of its $k$-dimensional cones and we will use $\Sigma(1)^{gen}$ for the set of ray generators. We use the letter $N$ for the character lattice instead of $M$ as our attention will be focused on the mirror dual side, where the roles of $N$ and $M$ are reversed.

We now recall the definition of a nef-partition. Let $\Delta \subset N_{\R}$ be a reflexive polytope, a decomposition $\Delta=\Delta_1+...+\Delta_k$ into a Minkowski sum of $k$ positive-dimensional (but not necessarily full-dimensional) polytopes, where $0_N \in \Delta_i$ for $i=1,...,k$, is called a nef-partition. The decomposition is denoted as $\Delta; \Delta_1,...,\Delta_k$. The identity of Minkowski sums translates to the identity of torus-invariant Cartier divisors: $D_{\Delta_1} + ... + D_{\Delta_k} = - K_{\P_{\Delta}}$. 

We define two divisors on $\P_{\Delta}$:
\begin{equation} \label{delta1-polytope-formula}
    D_{\Delta_1} = D_{e_1} + D_{e_2} + D_{e_3},
\end{equation}
\begin{equation} \label{delta2-polytope-formula}
    D_{\Delta_2} = D_{e_4} + D_{e_5} + D_{e_6}.
\end{equation}
These divisors correspond to 
$$V(x_1) + V(x_2)  + V(x_3),$$
$$V(x_4) + V(x_5)  + V(x_6).$$
Here, $V(x_i)$ refers to the coordinate hyperplane in $\P^5$ defined by the homogeneous equation $x_i=0$. Thus, $D_{\Delta_1}$ and $D_{\Delta_2}$ are clearly basepoint-free Cartier divisors and $D_{\Delta_1} + D_{\Delta_2} = \sum_{\rho \in \Sigma_{\Delta}(1)} D_{\rho} = - K_{\P_{\Delta}}$. We remind the reader that the fundamental polytope of a torus-invariant Cartier divisor $D= \sum_{\rho \in \Sigma_{\Delta}(1)} a_{\rho} D_{\rho}$ is defined as
$$P_{D} = \{n \in N_{\R} : (n,u_\rho) + a_\rho \geq 0 \text{ for every } \rho \in \Sigma_{\Delta}(1)\},$$
where $u_{\rho}$ is the ray generator for a ray $\rho \in \Sigma(1)$, see \cite[Formula 6.1.1]{CLS}. The fundamental polytopes of the divisors $D_{\Delta_1}$ and $D_{\Delta_2}$ are given by
    $$\Delta_1 := P_{D_{\Delta_1}} = \conv\{u_{1},u_{2},...,u_{6}\} \subset N_{\R},$$
    $$\Delta_2 := P_{D_{\Delta_2}} = \conv\{v_{1},v_{2},...,v_{6}\} \subset N_{\R},$$
where
    \begin{equation}\label{u-definition}
        u_{i} = 3e_{i} - e_1 - e_2 - e_3, 1 \leq i \leq 6,
    \end{equation}
    \begin{equation}\label{v-definition}
        v_{i} = 3e_{i} - e_4 - e_5 - e_6, 1 \leq i \leq 6.
    \end{equation}
The polytope $\Delta$ is a reflexive $5$-dimensional polytope, and $\Delta_1$, $\Delta_2$ are full-dimensional polytopes such that $\Delta = \Delta_1 + \Delta_2$, so, $\Delta; \Delta_1, \Delta_2$ is a nef-partition. 
Following \cite[Definition 4.9]{BB}, we can construct a dual nef-partition $\nabla; \nabla_1, \nabla_2$, corresponding to the nef-partition constructed $\Delta; \Delta_1, \Delta_2$. The tuple of polytopes $\nabla; \nabla_1, \nabla_2$, defined by the formulas
\begin{equation}\label{nabla1-polytope-formula}
    \nabla_1 = \conv \{0,e_1,e_2,e_3\} \subset M_{\R},
\end{equation}
\begin{equation}\label{nabla2-polytope-formula}
    \nabla_2 = \conv \{0,e_4,e_5,e_6\} \subset M_{\R},
\end{equation}
\begin{multline}\label{nabla-polytope-formula}
    \nabla = \conv \{e_1,e_2,e_3,e_4,e_5,e_6,e_1+e_4,\\
    e_1+e_5,e_1+e_6,e_2+e_4, e_2+e_5,e_2+e_6,\\
    e_3+e_4,e_3+e_5,e_3+e_6\} \subset M_{\R}
\end{multline}
is a dual nef-partition to the nef-partition $\Delta; \Delta_1, \Delta_2$. We can also compute the normal fan $\Sigma_{\nabla}$, which is the following fan supported in $N_{\R}$:
$$\Sigma_{\nabla}(1)^{gen} = \{u_{1},...,u_{6},v_{1},...,v_{6}\}.$$
There are three types of maximal cones, which correspond to the points in \eqref{nabla-polytope-formula}:
    $$U_i = \cone(u_{1},...,\widehat{u_{i}},...,u_{6}), 1 \leq i \leq 3,$$
    $$V_j = \cone(v_{1},...,\widehat{v_{j}},...,v_{6}), 4 \leq j \leq 6,$$
    $$C_{i,j} = \cone(u_{1},...,\widehat{u_{i}}, ..., \widehat{u_{j}}, ..., u_{6}, v_{1}, ..., \widehat{v_{i}},
    ..., \widehat{v_{j}}, ..., v_{6}), 1 \leq i \leq 3, 4 \leq j \leq 6.$$
Here, the hat indicates that the corresponding vector is omitted. The vectors $u_i$ and $v_j$ are defined by the formulas \eqref{u-definition} and \eqref{v-definition}, respectively. Thus, we have
\begin{equation} \label{sigma5-nabla-gen}
    \Sigma_{\nabla}(5) = \{U_1,U_2,U_3\} \cup \{V_4,V_5,V_6\} \cup \{C_{i,j} : 1 \leq i \leq 3, 4 \leq j \leq 6\}.
\end{equation}

The last piece of information that we need is the torus-invariant basepoint-free Cartier divisors $D_{\nabla_1}, D_{\nabla_2}$. These basepoint-free divisors can be recovered from their fundamental polytope using the formula:
$$D_{P} = \sum_{\rho \in \Sigma(1)} (-\min_{m \in P \cap M} (u_{\rho},m))D_{\rho},$$
see \cite[Theorem 6.1.7(g)]{CLS}, \cite[Theorem 4.2.12(b)]{CLS}, and \cite[Theorem 6.3.12]{CLS}. Applying the formula above to our specific case, we obtain:
$$D_{\nabla_1} = D_{u_1} + ... + D_{u_6},$$
$$D_{\nabla_2} = D_{v_1} + ... + D_{v_6}.$$
We can summarize the above discussion in the following theorem:
\begin{proposition}\label{th:dual-toric-data}
    The polytopes $\nabla; \nabla_1,\nabla_2$, defined by formulas \eqref{nabla1-polytope-formula}, \eqref{nabla2-polytope-formula}, and \eqref{nabla-polytope-formula}, form the dual nef-partition to the nef-partition $\Delta; \Delta_1, \Delta_2$, where $\P_{\Delta} \cong \P^5$ and $D_{\Delta_1} = V(x_1) + V(x_2) + V(x_3), D_{\Delta_2} = V(x_4) + V(x_5) + V(x_6)$. The normal fan $\Sigma_{\nabla}$ has ray generators:
    $$\Sigma_{\nabla}(1)^{gen} = \{u_1,...,u_6,v_1,...,v_6\},$$
    where $u_i,v_j \in N$ are vectors defined by the formulas \eqref{u-definition} and \eqref{v-definition}. The set of maximal cones $\Sigma_{\nabla}(5)$ is given by the formula \eqref{sigma5-nabla-gen}. The basepoint-free Cartier divisors associated to the polytopes $\nabla_1$ and $\nabla_2$ are 
    $$D_{\nabla_1} = D_{u_1} + ... + D_{u_6},$$ $$D_{\nabla_2} = D_{v_1} + ... + D_{v_6}.$$
\end{proposition}

We now construct an MPCP desingularization of $\mathbb P_{\nabla}$, in the sense of \cite[Definition 2.2.13]{B}. In the toric setting, the construction described in \cite[Theorem 2.2.24]{B} reduces to choosing a maximal regular star triangulation of the polytope
$$P=\operatorname{conv}(\Delta_1,\Delta_2)
=\operatorname{conv}\{u_1,\ldots,u_6,v_1,\ldots,v_6\}
\subset N_{\mathbb R}.$$
We use the term triangulation in the standard sense: a finite collection of affinely embedded simplices whose union is P, closed under taking faces, and such that the intersection of any two simplices is a common face of both. A star triangulation is a triangulation in which each maximal simplex contains $0_N$ as a vertex. We construct a canonical maximal, regular, star triangulation of the polytope $P = \conv(\Delta_1, \Delta_2)$. 

In the following theorem we define 
$$\A^1 := \{(x_7,x_8) \in \R^2 : x_7 + x_8 = 1\},$$ 
and we remind the reader that $\Cayley(\Delta_1,\Delta_2)$ is a polytope with vertices $(a,1,0)$ and $(b,0,1)$, where $a$ is a vertex of $\Delta_1$ and $b$ is a vertex of $\Delta_2$. See also \cite{Triangulations} for the general information about Cayley polytopes and Cayley embeddings.

Let $\mathcal K\subset \mathbb R^{n'}$ and
$\mathcal L\subset \mathbb R^{n''}$ be polytopal complexes, and let
$\tau(\mathcal K)$ and $\tau(\mathcal L)$ be triangulations
of $\mathcal K$ and $\mathcal L$, respectively. Let
$
\pi:\mathbb R^{n'}\longrightarrow \mathbb R^{n''}
$
be an affine map. We say that $\tau(\mathcal L)$ is
$\pi$-induced from $\tau(\mathcal K)$ if there exists a simplicial
subcomplex
$
\tau'\subseteq \tau(\mathcal K)
$
such that $\pi$ induces an isomorphism of geometric simplicial complexes
$$
\tau'\xrightarrow{\sim}\tau(\mathcal L).
$$
Equivalently, $\pi$ restricts to an affine isomorphism on every simplex
of $\tau'$, the induced map
$$
|\tau'|\longrightarrow |\tau(\mathcal L)|
$$
is a homeomorphism, and
$
\tau(\mathcal L)
=
\{\pi(\sigma):\sigma\in\tau'\}.
$

\begin{theorem} \label{th:triangulation}
    Let $P = \conv(\Delta_1,\Delta_2) \subset N_{\R}$ be the face polytope of the fan $\Sigma_{\nabla}$ of a toric ambient space for the Batyrev--Borisov mirror dual of complete intersection of two cubics in $\P^5$. The natural lattice-affine projection $C : N_{\R} \times \mathbb A^1 \to N_{\R}$ restricts to the map of polytopes $C: \Cayley(\Delta_1, \Delta_2) \to \conv(\Delta_1,\Delta_2)$.  
    
    There exists a maximal regular star triangulation $\tau(P)$ of $P$, such that induced triangulation of the boundary $\tau(\p P)$ is $C$-induced from the maximal regular triangulation $\tau(\Cayley(\Delta_1, \Delta_2))$ with height function
    $$\omega(x) = \sum_{\substack{1 \leq i\leq j \leq 8\\ 0 \leq k \leq 3}}  |\sum_{i \leq t \leq j} y_t - k|, \quad y_i = x_i + \delta_i^1 x_7 + \delta_i^2 x_7 + \delta_i^3 x_7 + \delta_i^4 x_8 + \delta_i^5 x_8 + \delta_i^6 x_8,$$ 
    coming from the subdivision of $N_{\R}\times \mathbb A^1$ into affine Weyl chambers of type $A_6$. The symbol $\delta_{a}^b$ denotes the Kronecker symbol, which is equal to $1$ when $a=b$ and equal to $0$ otherwise. 
\end{theorem}
\begin{proof}
    Define the affine lattice 
    $$N' = \{(x_1,\ldots, x_6) \in \Z^6 : x_1 + \ldots + x_6 = 3\}.$$ 
    And consider the standard 3-dilated $5$-simplices in $N'_{\R}$:
    $$\Delta_1'= \Delta_2' = \conv \{3e_1,3e_2,3e_3,3e_4,3e_5,3e_6\} \subset N'_{\R}.$$
    So their Cayley polytope lies in $N'_{\R} \times \A^1$:
    $$\Cayley(\Delta_1',\Delta_2') \subset N'_{\R} \times \A^1.$$
    Consider the cumulative coordinates $z_0, z_1, \ldots, z_8$ on the $\R^8$, defined by  $z_i = x_1 + \ldots + x_i$. In particular $z_0 = 0$. Restricting the coordinates to the affine subspace 
    $$N'_{\R} \times \A^1 \subset \R^8$$
    is the same as imposing conditions $z_6 - z_0 = 3, z_8 - z_6 = 1$. The regular triangulation $\tau(\Cayley(\Delta_1',\Delta_2'))$ of $\Cayley(\Delta_1',\Delta_2')$ associated with the height function 
    $$\nu =  \sum_{\substack{1 \leq i\leq j \leq 8\\ 0 \leq k \leq 3}}  |\sum_{i \leq t \leq j} x_t - k|$$
    is the triangulation obtained by drawing all hyperplanes 
    $$H_{ij}^{k} = \{(x_1,\ldots,x_8) \in \R^8 : z_j - z_{i-1} = k\} \text{ for }1 \leq i \leq j \leq 8, 0 \leq k \leq 3$$ 
    and dissecting the polytope $\Cayley(\Delta_1',\Delta_2')$ along them. $\Cayley(\Delta_1',\Delta_2')$ can be defined by system of inequalities 
    $$z_i - z_{i-1} \geq 0 \text{ for } 1 \leq i \leq 8$$
    inside the affine subspace $N'_{\R} \times \A^1$. In other words, $\Cayley(\Delta_1',\Delta_2')$ is an alcoved polytope in the terminology of the \cite{LamPostnikovI}. So it is naturally equipped with triangulation $\tau(\Cayley(\Delta_1',\Delta_2'))$ into alcoves (or Weyl chambers), see \cite[Section 2.3]{LamPostnikovI} and \cite[Chapter III, Section 2]{KunMum}. These alcoves are of type $A_6$, since there are only $6$ effective cumulative coordinates.
    
    The lattice-unimodular affine transformation sending $e_7 \mapsto e_7 - e_1 - e_2 - e_3$ and $e_8 \mapsto e_8 - e_4 - e_5 - e_6$ maps $\Delta'_1 \times \{(1,0)\}$ onto $\Delta_1 \times \{(1,0)\}$ and $\Delta'_2 \times \{(0,1)\}$ onto $\Delta_2 \times \{(0,1)\}$, transforms the coordinates $x_i$ into the coordinates $y_i$ and the function $\nu$ to the function $\omega$. Here, $\Delta_1$ and $\Delta_2$ are defined by the formulas \eqref{delta1-polytope-formula} and \eqref{delta2-polytope-formula}.

To show that this induces a triangulation of $P=\operatorname{conv}(\Delta_1,\Delta_2)$, we construct an explicit polytopal subcomplex
$$
A\subset \operatorname{skel}_4\operatorname{Cayley}(\Delta_1,\Delta_2)
$$
whose support maps PL-homeomorphically onto $\partial P$ under $C$. Let $A$ be the pure $4$-dimensional polytopal subcomplex whose maximal faces are the following:
\begin{itemize}
    \item for $1\leq i\leq 3$, the face $(U_i\cap\partial P)\times\{(1,0)\},$ which is mapped by $C$ to $U_i\cap\partial P$;

    \item for $4\leq j\leq 6$, the face $(V_j\cap\partial P)\times\{(0,1)\},$ which is mapped by $C$ to $V_j\cap\partial P$;

    \item for $1\leq i\leq 3$ and $4\leq j\leq 6$, the Cayley face
    $$\Cayley \left(\operatorname{conv}\{u_k:k\neq i,j\}, \operatorname{conv}\{v_k:k\neq i,j\}\right),
    $$
    which is mapped by $C$ to $C_{i,j}\cap\partial P$.
\end{itemize}
By construction $C$ maps affine-isomorphically facets of $A$ to the facets of $\p P$ and respects intersections, so it is a PL-homeomorphism. The induced triangulation $\tau(A)$ then can be transferred to the corresponding triangulation $\tau(\p P)$ via this PL-homeomorphism.

This gives a maximal triangulation $\tau(\p P)$, which is regular on each facet. After adding $\{0\}$ to every maximal simplex, we obtain the maximal regular star triangulation $\tau(P)$.
\end{proof}

Now we have a maximal star triangulation $\tau(P)$, which is regular on each facet. More precisely, $\tau(P)$ is induced by the height function $\psi:P\cap N\to\mathbb R$,  defined on the lattice points by
$$
\psi(p)=
\begin{cases}
\omega(p_1+1,p_2+1,p_3+1,p_4,p_5,p_6,1,0),
& \text{if }p\text{ is of }u\text{-type},\\[1ex]
\omega(p_1,p_2,p_3,p_4+1,p_5+1,p_6+1,0,1),
& \text{if }p\text{ is of }v\text{-type},\\[1ex]
-M,
& p=0.
\end{cases}
$$
Here $M \gg 0$ is a large enough constant. Define the fan $\Pi$ as the face fan of the triangulation $\tau(P)$ --- that is, the maximal cones of $\Pi$ are in one-to-one correspondence with maximal simplices of $\tau(P)$. As stated in \cite[Section 4.2]{B}, the canonical morphism $\P_{\Pi} \to \P_{\Sigma_{\nabla}}$ is a crepant resolution (we also check this in the next section directly). We can summarize this in the following corollary:
\begin{corollary} \label{th:pi-description}
    There is a smooth complete fan $\Pi$ with ray generators in the lattice $N$ which is a subdivision of the fan $\Sigma_{\nabla}$. The fan $\Pi$ has $110$ rays, with the following set of ray generators:
    $$\Pi(1)^{\text{gen}} = \{u_{ijk} :  1 \leq i \leq j \leq k \leq 6, (i,j,k) \neq (1,2,3)\}$$
    $$\cup \{v_{ijk} :  1 \leq i \leq j \leq k \leq 6, (i,j,k) \neq (4,5,6)\},$$
    $$u_{ijk} = -e_1 - e_2 - e_3 + e_i + e_j + e_k, \qquad v_{ijk} = -e_4 - e_5 - e_6 + e_i + e_j + e_k,$$
    and the canonical morphism of the corresponding toric varieties $\pi : \P_{\Pi} \to \P_{\Sigma_{\nabla}}$ is a crepant resolution. $\P_{\Pi}$ is a smooth projective toric variety.
\end{corollary}

\section{Geometry of the mirror dual family} \label{sec:dual-family}
In this section we will construct the mirror family and establish its geometric properties. By construction, the fan $\Pi$ is a subdivision of the fan $\Sigma_{\nabla}$. Hence, the identity morphism $id : N_{\R} \to N_{\R}$ induces a natural morphism of fans $\Pi \to \Sigma_{\nabla}$. This gives a morphism of algebraic varieties, denoted by $\pi: \P_{\Pi} \to \P_{\Sigma_{\nabla}}$, which is a crepant resolution. One can compute the Cartier divisors $\pi^* D_{\nabla_1}$ and $\pi^* D_{\nabla_2}$ in terms of the torus-invariant divisors of the variety $\P_{\Pi}$:
    $$D_1 := \pi^* D_{\nabla_1} = \sum_{\substack{1 \leq i_1 \leq i_2 \leq i_3 \leq 6 \\(i_1,i_2,i_3) \neq (1,2,3)}} D_{u_{i_1i_2i_3}},\quad D_2 := \pi^* D_{\nabla_2} = \sum_{\substack{1 \leq i_1 \leq i_2 \leq i_3 \leq 6 \\ (i_1,i_2,i_3) \neq (4,5,6)}} D_{v_{i_1i_2i_3}},$$
using the standard techniques of \cite[Theorem 4.2.12]{CLS}.

As discussed in Section \ref{sec:dual-nef}, the variety $\P_{\Pi}$ is smooth and projective. Now, we introduce the family $\Y_{\xi_1,a_1,a_2,a_3,\xi_2,a_4,a_5,a_6} = V(s_1,s_2)$, where:

    $$s_1  = 3 \xi_1 \prod_{\substack{(i_1,i_2,i_3) \neq \\ (1,2,3)}} u_{i_1 i_2 i_3} - \sum_{t=1}^{3} a_t (\prod_{\substack{(i_1,i_2,i_3) \neq \\ (4,5,6)}} v_{i_1 i_2 i_3}^{\delta_{i_1}^t  + \delta_{i_2}^t + \delta_{i_3}^t}  \prod_{\substack{(i_1,i_2,i_3) \neq \\ (1,2,3)}} u_{i_1 i_2 i_3}^{\delta_{i_1}^t  + \delta_{i_2}^t + \delta_{i_3}^t}),$$

    $$s_2 = 3 \xi_2 \prod_{\substack{(i_1,i_2,i_3) \neq \\ (4,5,6)}} v_{i_1 i_2 i_3} - \sum_{t=4}^{6} a_t (\prod_{\substack{(i_1,i_2,i_3) \neq \\ (4,5,6)}} v_{i_1 i_2 i_3}^{\delta_{i_1}^t  + \delta_{i_2}^t + \delta_{i_3}^t}  \prod_{\substack{(i_1,i_2,i_3) \neq \\ (1,2,3)}} u_{i_1 i_2 i_3}^{\delta_{i_1}^t + \delta_{i_2}^t + \delta_{i_3}^t}). $$

The notation $V(s_1,s_2)$ denotes the scheme-theoretic zero locus of the sections $s_1,s_2$. Each product is taken over all indices $1 \leq i_1 \leq i_2 \leq i_3 \leq 6$, except for those triples that are explicitly excluded. We consider a family of varieties defined as the scheme-theoretic intersection of the zero sets of two sections, $s_1$ and $s_2$, chosen from $H^0(\mathcal O (D_1) )$ and $H^0(\mathcal O(D_2))$, respectively. The equation is written in affine coordinates on the Cox space $\C^{\Pi(1)} - Z(\Pi)$, where $Z(\Pi)$ is the toric irrelevant locus.

The symbols $u_{i_1i_2i_3}, v_{i_1i_2i_3}$ are used both for ray generators (elements of $\Pi(1)^{gen}$) and for corresponding variables in the Cox ring $\C[r : r \in \Pi(1)^{gen}]$, while all other symbols are merely parameters of the family $\Y$. 

As stated in \cite{BB}, the family $\Y_{\xi_1,...,a_6}$ is a mirror dual family to the complete intersection of two cubic hypersurfaces in $\P^5$. Specifically, the Hodge diamond of the smooth complete intersection of two cubics is mirror-symmetric to the Hodge diamond of a sufficiently generic fiber of the family $\Y_{\xi_1,...,a_6}$. 

Hence, for sufficiently generic parameters, we have $h^{2,1} (\Y_{\xi_1,...,a_6}) = 1$ (see \cite[Theorem 4.15]{StringyBB}). The number $h^{2,1}$ can be interpreted as the dimension of the local deformation space of complex structures. Although the family $\Y_{\xi_1,...,a_6}$ is $6$-dimensional, we expect it to be effectively a $1$-dimensional family. The overabundance hints that the family might be parameterized inefficiently, and that many fibers could actually be isomorphic. This is indeed the case: up to automorphisms, the family depends only on the value $\frac{a_1 ... a_6}{\xi_1^3 \xi_2^3}$, at least whenever $a_i \neq 0, \xi_i \neq 0$. 

So we define the family $\Y_{\psi} = V(h_1,h_2)$, where $h_1$ and $h_2$ are defined by:
\begin{equation} \label{eq:h1}
    h_1  = 3 \psi \prod_{\substack{(i_1,i_2,i_3) \neq \\ (1,2,3)}} u_{i_1 i_2 i_3} - \sum_{t=1}^{3} (\prod_{\substack{(i_1,i_2,i_3) \neq \\ (4,5,6)}} v_{i_1 i_2 i_3}^{\delta_{i_1}^t  + \delta_{i_2}^t + \delta_{i_3}^t}  \prod_{\substack{(i_1,i_2,i_3) \neq \\ (1,2,3)}} u_{i_1 i_2 i_3}^{\delta_{i_1}^t  + \delta_{i_2}^t + \delta_{i_3}^t}),
\end{equation}
\begin{equation} \label{eq:h2}
    h_2 = 3 \psi \prod_{\substack{(i_1,i_2,i_3) \neq \\ (4,5,6)}} v_{i_1 i_2 i_3} - \sum_{t=4}^{6} (\prod_{\substack{(i_1,i_2,i_3) \neq \\ (4,5,6)}} v_{i_1 i_2 i_3}^{\delta_{i_1}^t  + \delta_{i_2}^t + \delta_{i_3}^t}  \prod_{\substack{(i_1, i_2, i_3) \neq \\ (1,2,3)}} u_{i_1 i_2 i_3}^{\delta_{i_1}^t + \delta_{i_2}^t + \delta_{i_3}^t}), 
\end{equation}
where the parameter $\psi$ belongs to $\P^1$, the total space is $\Y \subset \P_{\Pi} \times \P^1$, and the structural morphism is $f : \Y \to \P^1$. Since $\Y_{\psi}$ is a subfamily of $\Y_{\xi_1,...,a_6}$, we continue to use the symbol $\Y$ to denote both the total space of the new family and the family itself. 

It is also useful to write down the equations in toric coordinates:
\begin{equation} \label{eq:h-inside-torus}
    h_1|_{\T_N} = 3 \psi - t_1 - t_2 - t_3, \quad h_2|_{\T_N} = 3 \psi - t_4 - t_5 - t_6.
\end{equation}
Thus, we have $\Y_{\psi} \cap \T_N = V(h_1|_{\T_N},h_2|_{\T_N}) \subset \T_N$. 

In the further considerations we will also need the family $\X_{\psi}$, with fibers $\X_{\psi} \subset \P_{\Sigma_{\nabla}}$, which lives in $\P_{\Sigma_{\nabla}}$ and corresponds to the family $\Y_{\psi}$ before the crepant desingularization procedure. So, let $\X_{\psi} = V(r_1, r_2)$, where 
$$r_1 = 3 \psi u_1 ... u_6 - u_1^3 v_1^3 - u_2^3 v_2^3 - u_3^3 v_3^3, \qquad r_2 = 3 \psi v_1 ... v_6 - u_4^3 v_4^3 - u_5^3 v_5^3 - u_6^3 v_6^3,$$ 
$$r_1 \in \Gamma(\P_{\Sigma_{\nabla}}, \O(D_{\nabla_1})), \qquad r_2 \in \Gamma(\P_{\Sigma_{\nabla}}, \O(D_{\nabla_2})).$$
These equations are expressed in terms of the Cox ring of the variety $\P_{\Sigma_{\nabla}}$. For the definitions of the fan $\Sigma_{\nabla}$ and the divisors $D_{\nabla_1}, D_{\nabla_2}$ see Theorem \ref{th:dual-toric-data}. The canonical map 
$$\pi : \P_{\Pi} \to \P_{\Sigma_{\nabla}}$$
restricts to the canonical map $\pi : \Y_{\psi} \to \X_{\psi}$ for every $\psi \in \P^1$. The following proposition shows that $\X_{0}$ is normal.

\begin{proposition}\label{th:X0-normal}
    The total space $\X \subset \P_{\Sigma_{\nabla}} \times \P^1$ is a normal variety. The fiber $\X_0 \subset \P_{\Sigma_{\nabla}}$ is a normal variety. The variety $\X_0$ is integral.
\end{proposition}
\begin{proof}
    To prove the first statement, we use Serre's normality criterion, which states that $\X$ is normal if and only if it satisfies conditions $S_2$ and $R_1$. For the $S_2$ condition, note that $\X$ is even Cohen--Macaulay since $\P_{\Sigma_{\nabla}} \times \P^1$ is, and $\X$ is a complete intersection of dimension $4$ inside $\P_{\Sigma_{\nabla}} \times \P^1$. 
    
    Due to the symmetry of the problem, it is enough to compute $(\spec \C[U_{\sigma}] \times \P^1) \cap \X$ only in the cases when $\sigma$ is $U_1$ or $C_{3,6}$. To verify the $R_1$ condition, we need to prove that the singular locus of the following four affine algebraic $\C$-schemes has dimension at most $2$:
    $$\spec \frac{\C[y_1,...,y_6, \psi]}{(y_2 y_3 y_4 y_5 y_6 - y_1^3, 3\psi y_1 - 1 - y_2 - y_3, 3\psi  - y_4 - y_5 - y_6)},$$
    $$ \spec \frac{\C[y_1,...,y_6, \psi]}{(y_1 y_2 y_4 y_5 - y_3^3 y_6^3, 3 \psi y_3 - y_1 - y_2 - 1, 3 \psi y_6 - y_4 - y_5 - 1)},$$
    $$ \spec \frac{\C[y_1,...,y_6, \psi^{-1}]}{(y_2 y_3 y_4 y_5 y_6 - y_1^3, 3 y_1 - \psi^{-1}( 1 + y_2 + y_3), 3  - \psi^{-1}( y_4 + y_5 + y_6))},$$
    $$ \spec \frac{\C[y_1,...,y_6, \psi^{-1}]}{(y_1 y_2 y_4 y_5 - y_3^3 y_6^3, 3 y_3 - \psi^{-1}( y_1 + y_2 + 1), 3 y_6 - \psi^{-1} (y_4 + y_5 + 1))}.$$
    For each of these affine schemes, the singular locus is defined by the ideal generated by the three defining equations together with the $3\times 3$ minors of the Jacobian matrix. Computing this ideal, we find that the corresponding singular loci have dimension $2$. Therefore $\X$ is regular in codimension one. Hence $\X$ satisfies $R_1$, and so $\X$ is normal.    
    
    To show that $\X_0$ is normal, we again need to verify the $S_2$ and $R_1$ conditions. The $S_2$ condition is satisfied because $\P_{\Sigma_{\nabla}}$ is Cohen--Macaulay and $\X_0$ is a complete intersection of dimension $3$. The $R_1$ condition can be checked by computing the dimension of the singular locus of:
    $$ \spec \frac{\C[y_1,...,y_6]}{(y_2 y_3 y_4 y_5 y_6 - y_1^3, - 1 - y_2 - y_3,  - y_4 - y_5 - y_6)},$$
    $$\spec \frac{\C[y_1,...,y_6]}{(y_1 y_2 y_4 y_5 - y_3^3 y_6^3, - y_1 - y_2 - 1, - y_4 - y_5 - 1)}.$$
    The Jacobian computation shows that the singular loci of these two affine schemes have dimension at most $1$. Since $\dim \X_0=3$, this proves that $\X_0$ is regular in codimension one. Therefore $\X_0$ satisfies $R_1$, and hence is normal by Serre's criterion.

    For every ray $\rho \in \Sigma_\nabla(1)$, the variety
    $$\X_0\cap D_\rho$$
    has dimension at most $2$.
    Hence no irreducible component of $\X_0$ is contained in the toric boundary
    $\P_{\Sigma_\nabla}\setminus \T_N$.
    Therefore every irreducible component of $\X_0$ meets $\T_N$.
    Since $\X_0 \cap \T_N$ is irreducible, $\X_0$ has only one irreducible component. As $\X_0$ is normal, it is reduced. Thus $\X_0$ is reduced and irreducible, hence integral.
\end{proof}

The remaining section focuses on computing the smooth fiber locus of the family $\Y_{\psi}$. For clarity, we introduce the notation $\zeta_6 = \exp(\frac{2 \pi i}{6})$ to denote a primitive sixth root of unity, and $\mu_6 = \{1,\zeta_6, ..., \zeta_6^5\}$ for the set of all sixth roots of unity. Additionally, we classify the singularities of the singular fibers. We will summarize all the propositions about the family $\Y_{\psi}$ which we aim to prove in a single theorem, both for the sake of readability and to facilitate easier referencing:  

\begin{theorem} \label{th:geometry-summary}
    Let $\Y_{\psi} = V(h_1, h_2)$ be a family with $\psi \in \P^1$ and total space $\Y \subset \P_{\Pi} \times \P^1$. The equations $h_1$ and $h_2$ are defined by formulas \eqref{eq:h1}, \eqref{eq:h2}.
    \begin{enumerate}
        \item The family $\Y\to \P^1$ is flat over the one-dimensional base $\P^1$.
        \item The total space of the family $\Y \subset \P_{\Pi} \times \P^1$ is smooth, except for the points lying in $\Y_\infty \subset \Y$.
        \item When $\psi \not \in \mu_6 \cup \{0,\infty\}$, the fiber $\Y_{\psi}$ is smooth.
        \item When $\psi \in \mu_6$, the fiber $\Y_{\psi}$ has a unique ordinary double point, which lies in $\Y_{\psi} \cap \T_N$ and has coordinates $(\psi,...,\psi) \in \T_N$.
        \item \label{thm:geometrysummary-item5} When $\psi = 0$, the fiber $\Y_{0} = W_0 \cup W_1 \cup W_2$ has three integral components, which are defined by the equations:
        $$W_0 = V\left(\frac{h_1|_{\psi = 0}}{v_{123}}, \frac{h_2|_{\psi = 0}}{u_{456}}\right), W_1 = V\left(v_{123}, \frac{h_2|_{\psi = 0}}{u_{456}}\right), W_2 = V\left(\frac{h_1|_{\psi = 0}}{v_{123}},  u_{456}\right).$$
        $W_1, W_2$ are smooth rational exceptional components, whereas $W_0$ is the strict transform of $\X_0$.
    \end{enumerate}
\end{theorem}
    Point (1) follows from miracle flatness, since $\Y$ is Cohen-Macaulay, $\P^1$ is regular and every fiber has dimension $3$. Points (2)--(4) follow directly from the Jacobian criterion. We therefore focus on point (5). If $\psi=0$, then $h_1$ is divisible by $v_{123}$ and $h_2$ is divisible by $u_{456}$, as follows from the equations \eqref{eq:h1}, \eqref{eq:h2}. We introduce the notation
    \begin{equation}\label{def:th}
    \th_1 := \frac{h_1|_{\psi=0}}{v_{123}},
    \qquad
    \th_2 := \frac{h_2|_{\psi=0}}{u_{456}}.
    \end{equation}
    With this notation, we have
    $$
    W_0 = V(\th_1,\th_2), \qquad
    W_1 = V(v_{123},\th_2), \qquad
    W_2 = V(\th_1,u_{456}).
    $$
    We will also use $\phi_{i_0} : \Pi(1)^{gen} \to \Z$ for the function defined by $$\phi_{i_0}(u_{ijk}) = \delta_{i_0}^i + \delta_{i_0}^j + \delta_{i_0}^k, \quad \phi_{i_0}(v_{ijk}) = \delta_{i_0}^i + \delta_{i_0}^j + \delta_{i_0}^k.$$
         
\begin{proposition} \label{th:w1-w2-integrality}
    The components $W_1$ and $W_2$ are smooth, integral and rational varieties.
\end{proposition}
\begin{proof}
    We now prove that $W_1 = V(v_{123}, \th_2)$ and 
    $W_2 = V(\th_1, u_{456})$ are smooth. The only way $W_2$ could be singular is if there exists a point $p \in V(\th_1, u_{456})$ such that $(\grad \th_1)|_p = 0$. This follows from the Jacobian criterion and the observation that $(\grad u_{456})|_p$ (here, $u_{456}$ considered as monic polynomial) has a $1$ in the position corresponding to the variable $u_{456}$, while $(\grad \th_1)|_p$ has a $0$ in that same position. Therefore, the gradients are not collinear-by-nonzero-scalars, so, we should have $(\grad \th_1)|_p = 0$. The existence of such a point $p$ would imply the existence of a cone $\sigma \in \Pi(5)$ with the following properties:
    \begin{itemize}
        \item $u_{456} \in \sigma$, which ensures that $W_2 \cap U_{\sigma} \neq \emptyset$.
        \item There is a triple of not necessarily distinct $u$-ray generators $r_1, r_2, r_3 \in \sigma^{gen} - \{u_{456}\}$ such that $\phi_1(r_1) \geq 1, \phi_2(r_2) \geq 1$, and $\phi_3(r_3) \geq 1$. This follows from the identities
        $$\frac{\partial \th_1}{\partial v_{111}} \big|_p = \frac{\partial \th_1}{\partial v_{222}} \big|_p = \frac{\partial \th_1}{\partial v_{333}} \big|_p = 0.$$
    \end{itemize} 
    One can check that no such cone exists: the first condition forces $\sigma$ to be a subcone of $U_1, U_2,$ or $U_3$ and the second condition eliminates even those possibilities. The smoothness of $W_1$ can be verified using the same argument.

    We prove that $W_1$ and $W_2$ are integral and rational. By the argument above, both $W_1$ and $W_2$ are smooth. It remains to prove that they are connected. We treat $W_1$; the argument for $W_2$ is symmetric. Consider the cone
    $$
    \sigma=\operatorname{cone}\{v_{123},v_{124},v_{224},v_{234},v_{235}\}\in \Pi.
    $$
    On the corresponding affine toric chart $U_\sigma$, all Cox coordinates corresponding to rays not contained in $\sigma$ are set equal to $1$. Therefore
    $$
    U_\sigma\cap W_1=U_\sigma\cap V(v_{123},\th_2)=V(v_{123},v_{124}v_{224}v_{234}+v_{235}+1)\subset U_\sigma.
    $$
    Since
    $$
    U_\sigma \cong \spec\mathbb C[v_{123},v_{124},v_{224},v_{234},v_{235}],
    $$
    the equation $v_{123}=0$ eliminates $v_{123}$, and the second equation solves uniquely for $v_{235}$. Hence
    $$
    U_\sigma\cap W_1\cong \spec\mathbb C[v_{124},v_{224},v_{234}]\cong \mathbb C^3.
    $$
    In particular, $U_\sigma\cap W_1$ is connected and irreducible. We claim that the complement $W_1\setminus (U_\sigma\cap W_1)$ has dimension at most $2$. Indeed, the complement of $U_\sigma$ in $\P_\Pi$ is a union of toric boundary divisors $D_\rho$, with $\rho\notin\sigma(1)$. Hence it is enough to check that
    $$
    \dim(W_1\cap D_\rho)\leq 2
    $$
    for every such ray $\rho$. This is verified on affine toric charts by adding the equation of $D_\rho$ to the defining equations $v_{123}=0$ and $\th_2=0$ of $W_1$; the resulting schemes all have dimension at most $2$. Since $W_1$ is smooth of dimension $3$, every connected component of $W_1$ is open and has dimension $3$. Therefore no connected component can be contained in $W_1\setminus (U_\sigma\cap W_1)$. Hence every connected component of $W_1$ meets $U_\sigma\cap W_1$. Since $U_\sigma\cap W_1$ is connected, $W_1$ is connected. The same argument proves that $W_2$ is connected. Thus $W_1$ and $W_2$ are smooth and connected, hence reduced and irreducible. Therefore $W_1$ and $W_2$ are integral. The same affine charts also show that both components are rational.
\end{proof}

\begin{proposition} \label{th:y0-decomposition}
The central fiber $\Y_0$ decomposes scheme-theoretically as
$$
\Y_0=W_0\cup W_1\cup W_2.
$$
The component $W_0$ is integral and the strict transform of $\X_0$.
\end{proposition}
\begin{proof}
Let
$$\pi:\P_\Pi\to \P_{\Sigma_\nabla}$$
be the toric morphism induced by the refinement $\Pi$ of $\Sigma_\nabla$. It restricts to a proper morphism
$$\pi_0:\Y_0\to \X_0.$$

We now prove that $W_0$ is the strict transform of $\X_0$. Let $E=\operatorname{Exc}(\pi)$. On $\P_\Pi\setminus E$, the morphism $\pi$ is an isomorphism. Since $v_{123}$ and $u_{456}$ are invertible there, the equations
$$
h_1|_{\psi=0}=v_{123}\th_1,\qquad h_2|_{\psi=0}=u_{456}\th_2
$$
are equivalent on $\P_\Pi\setminus E$ to $\th_1=0$ and $\th_2=0$. Hence $W_0$ agrees with $\pi^{-1}(\X_0)$ away from $E$. We next show that $W_0$ has no irreducible component contained in $E$. Since $W_0$ is three-dimensional, it is enough to check that
$$
\dim V(x_\rho,\th_1,\th_2)\leq 2
$$
for every exceptional ray generator $\rho\in\Pi(1)^{gen}\setminus\Sigma_\nabla(1)^{gen}$, where $x_\rho$ denotes the Cox coordinate corresponding to $\rho$. This is verified on the affine toric charts. Thus no irreducible component of $W_0$ is contained in $E$.

We prove next that $W_0$ is reduced. Since $W_0=V(\th_1,\th_2)$ is a complete intersection in the smooth toric variety $\P_\Pi$, it is Cohen--Macaulay. Therefore it is enough to show that it is generically reduced. By the previous paragraph, every generic point of $W_0$ lies outside $E$. On $\P_\Pi\setminus E$, the scheme $W_0$ agrees with $\pi^{-1}(\X_0)$. Since $\X_0$ is normal and integral by Proposition~\ref{th:X0-normal}, it is regular at its generic point. Thus $W_0$ is regular at its generic points, hence generically reduced. Being Cohen--Macaulay, $W_0$ is reduced. Therefore $W_0$ is the strict transform of $\X_0$, where the strict transform is taken with its reduced induced scheme structure. Since $\X_0$ is integral by Proposition~\ref{th:X0-normal}, its strict transform $W_0$ is integral.

It remains to prove the scheme-theoretic decomposition of $\Y_0$. Since
$$
h_1|_{\psi=0}=v_{123}\th_1,\qquad h_2|_{\psi=0}=u_{456}\th_2,
$$
we have
$$
\Y_0=V(v_{123}\th_1,u_{456}\th_2).
$$
Set-theoretically,
$$
V(v_{123}\th_1,u_{456}\th_2)
=
V(\th_1,\th_2)\cup V(v_{123},\th_2)\cup V(\th_1,u_{456})\cup V(v_{123},u_{456}).
$$
The last term is empty because there is no cone in $\Pi$ containing both
$v_{123}$ and $u_{456}$. Hence
$$
|\Y_0|=|W_0|\cup |W_1|\cup |W_2|.
$$

We now check the equality scheme-theoretically. Since
$V(v_{123},u_{456})=\emptyset$, the open subsets $D(v_{123})$ and
$D(u_{456})$ cover $\P_\Pi$. On $D(v_{123})$, the variable $v_{123}$ is
invertible, so
$$
(v_{123}\th_1,u_{456}\th_2)=(\th_1,u_{456}\th_2).
$$
Because $u_{456}$ does not appear in $\th_1$ and $\th_2$, we have
$$
(\th_1,u_{456}\th_2)=(\th_1,\th_2)\cap(\th_1,u_{456}).
$$
Thus, on $D(v_{123})$, $\Y_0$ is the scheme-theoretic union of $W_0$ and
$W_2$. Similarly, on $D(u_{456})$, the variable $u_{456}$ is invertible and
$$
(v_{123}\th_1,u_{456}\th_2)=(v_{123}\th_1,\th_2)
=(\th_1,\th_2)\cap(v_{123},\th_2),
$$
because $v_{123}$ does not appear in $\th_1$ and $\th_2$. Hence, on
$D(u_{456})$, $\Y_0$ is the scheme-theoretic union of $W_0$ and $W_1$. Since
$D(v_{123})$ and $D(u_{456})$ cover $\P_\Pi$, the decomposition
$$
\Y_0=W_0\cup W_1\cup W_2
$$
holds scheme-theoretically.
\end{proof}

\section{Picard--Fuchs equation and infinity} \label{sec:inf}
We begin by recalling the geometric setup. There is a smooth projective toric variety $\P_{\Pi}$, associated with a smooth complete fan $\Pi$. The support of $\Pi$ lies in the real vector space $N_{\R} = N \otimes \R$, and there is a canonical embedding of the torus $\T_N = N \otimes_{\Z} \mathbb \C^{\times}$ into $\P_{\Pi}$. We consider a family of projective varieties $\Y_{\psi}$ over $\P^1$, with total space $\Y \subset \P_{\Pi} \times \P^1$, and structural morphism $f : \Y \to \P^1$. Each fiber $\Y_{\psi}$ is a projective subvariety of $\P_{\Pi}$. We denote by $U = \P^1 -  \{0,\infty, 1,\zeta_6,...,\zeta_6^5\}$ the open subset where the fibers are smooth, where $\zeta_6 = \exp(2 \pi i/6)$ is a primitive sixth root of unity. The family $\Y_{\psi} = V(h_1,h_2)$ is defined by equations \eqref{eq:h1} and \eqref{eq:h2}. We write $Y_{1,\psi} = V(h_1), Y_{2,\psi} = V(h_2)$ for the corresponding hypersurfaces in $\P_{\Pi}$, so by definition, $\Y_{\psi} = Y_{1,\psi} \cap Y_{2,\psi}$.

We will construct a homological class $\gamma(\psi) \in H_3(\Y_{\psi})$ (see Corollary \ref{cycle-lemma-family}) such that 
$$\tube_{1,\psi} \circ \tube_{2,\psi} (\gamma(\psi)) = (S^1)^5,$$ 
where $(S^1)^5 \in H_5(\T_N - (Y_{1,\psi} \cup Y_{2,\psi}))$ denotes a homological class of a standard cycle defined by the equations $|t_1|=...=|t_5|=|t_6|=1$ in the torus $\T_N$. See equations \eqref{eq:tube1} and \eqref{eq:tube2} for a description of the corresponding tube maps. This cycle will allow us to construct a global top-form $\omega = \omega(\psi) \in H^3(\Y_{\psi},\C)$, which varies holomorphically with the parameter $\psi$.

We utilize the toric Euler form $\Omega$, described, for example, in \cite[p369]{CLS}. The form is defined up to a sign $\pm 1$, but once an ordering of the basis $t_1,...,t_5 \in M$ is fixed, this ambiguity is resolved. We choose the ordering as written above. The form $\Omega$ is a differential $5$-form on the Cox space $\C^{\Pi(1)} - Z(\Pi)$. After dividing by a polynomial of the appropriate Cox degree, the form descends to a differential form on $\P_{\Pi}$ with meromorphic poles. Consider the following family over $\P^1$ of meromorphic differential forms on $\P_\Pi$:
\begin{equation} \label{etacirc}
    \omega^{\circ}(\psi) = \frac{(3\psi)^2 \Omega}{h_1 h_2}.
\end{equation}
The factor $(3\psi)^2$ is chosen so that the resulting residue form extends across $\psi=\infty$. The restriction of this form on the standard open torus $\T_N - (Y_{1,\psi} \cup Y_{2,\psi})$ is
\begin{equation} \label{eq:etacirc-def}
    \omega^{\circ}|_{\T_N - (Y_{1,\psi} \cup Y_{2,\psi})} = \frac{(3\psi)^2}{(3 \psi - t_4 - t_5 - t_6)(3 \psi - t_1 - t_2 - t_3)} \frac{dt_1}{t_1} \wedge ... \wedge \frac{dt_5}{t_5},
\end{equation}
where $t_1,\dots,t_5,t_6$ are standard characters on the torus $\T_N$; see equation \eqref{eq:m-ti-coordinates}.  
The fiber $\omega^{\circ}(\psi)$ is a global section of the sheaf $K_{\P_{\Pi}}(Y_{1,\psi} + Y_{2,\psi})$, where $Y_{1,\psi} = V(h_1)$ and $Y_{2,\psi} = V(h_2)$ are hypersurfaces in $\P_{\Pi}$, and $h_1, h_2$ are defined by equations \eqref{eq:h1}, \eqref{eq:h2}. Here, by $K_{\P_{\Pi}}$ we mean the canonical sheaf, also we will use $K_{\Y/\P^1}$ for the relative canonical sheaf of the family $\Y_{\psi}$. 

\begin{equation}\label{eq:omega-def}
    \omega(\psi) = \frac{1}{(2 \pi i)^3} \operatorname{res}_{2,\psi} \operatorname{res}_{1,\psi} \omega^{\circ}(\psi).
\end{equation} 
For each fixed $\psi \in U$, the residue maps
    $$\operatorname{res}_{1,\psi} : H^5(\P_{\Pi} - (Y_{1,\psi} \cup Y_{2,\psi})) \to H^4(Y_{1,\psi} - \Y_{\psi}),$$
    $$\operatorname{res}_{2,\psi} : H^4(Y_{1,\psi} - \Y_{\psi}) \to H^3(\Y_{\psi})$$
are adjoint to the tube maps
\begin{equation} \label{eq:tube2}
\operatorname{tube}_{1,\psi} : H_4(Y_{1,\psi} - \Y_{\psi}) \to H_5(\P_{\Pi} - (Y_{1,\psi} \cup Y_{2,\psi})),
\end{equation}
\begin{equation} \label{eq:tube1}
\operatorname{tube}_{2,\psi} : H_3(\Y_{\psi}) \to H_4(Y_{1,\psi} - \Y_{\psi}).
\end{equation}
The following lemma shows that these fiberwise forms arise from a single global section of the relative canonical sheaf. In particular, it justifies the normalization by the factor $(3\psi)^2$ and shows that the resulting relative form extends across $\psi=\infty$. 

\begin{lemma} \label{lemma:omegaconstruction} The section $\psi \mapsto \omega(\psi)$ corresponds to a global section of $f_\ast K_{\Y/\P^1}$. 
\end{lemma} 
\begin{proof} 
Consider the form $\omega^{\circ} \wedge d\psi$. Since $d\psi$ is a (trivializing) section of $K_{\P^1}(2[\infty])$, this corresponds to a top-form on $\P_{\Pi}\times \P^1$ with simple poles along $Y_1$ and $Y_2$, and a pole of order 2 at $\infty$, hence of $K_{\P_\Pi \times \P^1}(Y_1+Y_2)(2[\infty]).$ Restricting this sheaf to $\Y$ and applying adjunction provides a natural sequence of isomorphisms 
\begin{equation}\label{eq:adjunction} K_{\P_\Pi \times \P^1}(Y_1+Y_2)|_{\Y}(2[\infty]) \cong K_\Y(2[\infty]) \cong K_{\Y/\P^1} \otimes K_{\P^1}(2\cdot [\infty]) \cong  K_{\Y/\P^1}.
\end{equation}

The isomorphism \eqref{eq:adjunction} is locally given by sending a top-form in $K_\Y$, to  a section of $K_{\Y/\P^1}$, according to the rule:
$$\eta = \eta' \wedge d\psi \mapsto [\eta'].$$
This description makes sense around points where $\Omega_{\Y/\P^1}$ is locally free, i.e. the smooth points of $\Y_\psi$, and one sees by inspection that it is given by \eqref{eq:omega-def} for all $\psi$ with smooth $\Y_\psi$.
\end{proof}

\subsection{Existence of a cycle}
The goal of this subsection  (see Corollary \ref{cycle-lemma-family}) is to construct a homological cycle $\gamma(\psi) \in H_3(\Y_{\psi})$ such that 
$$\tube_{1,\psi} \circ \tube_{2,\psi} (\gamma(\psi)) = (S^1)^5,$$ 
where $(S^1)^5 \in H_5(\T_N - (Y_{1,\psi} \cup Y_{2,\psi}))$ denotes a standard cycle defined by the equations $|t_1|=...=|t_5|=|t_6|=1$ in the torus $\T_N$. The proof of the existence of such a cycle appears to be omitted from the related literature, so we include it for the sake of completeness.

In the following lemma, we will study the cycle $(S^1)^5$ in $\T_N \subseteq \P_\Pi$ given by $|t_i|=1$, where $t_i$ are the toric coordinates of $\T_N$. For $|\psi| > 1$, neither of the equations in \eqref{eq:h-inside-torus} have solutions with $|t_i|=1$, so in fact $(S^1)^5 \subseteq \T_N \setminus (Y_{1,\psi}\cup Y_{2, \psi})$. 

\begin{lemma}\label{tube-lemma}
For every $\psi \in U$, the following diagram    
{\small
$$
\begin{tikzcd}[column sep=small]
H_6(\P_{\Pi}-\Y_{\psi}) \arrow[rrrr, "\delta"] &  &  &  & H_5(\P_{\Pi}-(Y_{1,\psi} \cup Y_{2,\psi})) \\
  &  &  & H_4(Y_{1,\psi}-\Y_\psi) \arrow[ru, "\operatorname{tube}_{1,\psi}"] &  \\
  &  & H_3(\Y_{\psi}) \arrow[lluu, "\operatorname{tube}_{\psi}"] \arrow[ru, "\operatorname{tube}_{2,\psi}"] &  &                      \end{tikzcd}
$$}
commutes. Here, $\delta$ is the Mayer-Vietoris homomorphism, and $\Y_{\psi} = Y_{1,\psi} \cap Y_{2,\psi}$ is the smooth complete intersection of the smooth hypersurfaces $Y_{1,\psi}$ and $Y_{2,\psi}$.
\end{lemma}
\begin{proof}
We will work in the analytic category and verify the statement locally at the level of chains. The global statement then follows from the fact that the normal bundle $N_{\Y_{\psi}}$ splits as 
$$N_{\mathcal \Y_\psi/\mathbb P_\Pi} \cong N_{\Y_\psi/Y_{1,\psi}} \oplus N_{Y_{1,\psi}/\mathbb P_\Pi}|_{\Y_\psi}.$$

Fix $\psi \in U$. Let $\C^5 \cong U \subset \P_{\Pi}$ be a local patch such that $U \cap Y_{1,\psi}, U \cap Y_{2,\psi}$ are locally defined by equations $z_1 = 0$, $z_2 = 0$ respectively. Let $c \in Z_3(\Y_{\psi})$ be any $3$-cycle. We have:

$$\tube_{2,\psi}(c) \cap U = \{z_1 = 0, |z_2| = 1\} \times c|_U \subset \{0\} \times S^1 \times \C^3,$$
$$\tube_{1,\psi} \circ \tube_{2,\psi} (c) \cap U = \{|z_1| = |z_2| = 1\} \times c|_U  \subset S^1 \times S^1 \times \C^3,$$
$$\tube_{\psi}(c) = \{|z_1| + |z_2| = 2\} \times c|_U \subset B_2^1 \times B_2^1 \times \C^3.$$

Here, $B_2^1$ denotes the ball in $\C$ with radius $2$. We also think of $c \cap U$ as a subset of $\C^3$. In these local coordinates, define two chains:
$$\Sigma_1 = \{|z_1| \geq 1, |z_1| + |z_2| = 2\} \times c|_U, \qquad \Sigma_2 = \{|z_2| \geq 1, |z_1| + |z_2| = 2\} \times c|_U.$$
Note that $\tube_{\psi}(c) = \Sigma_1 + \Sigma_2$ so we see that $\delta \tube_{\psi}(c) = \partial \Sigma_1 = \tube_{1,\psi} \circ \tube_{2, \psi}(c)$
and the claim follows.
\end{proof}

\begin{lemma} \label{topology-lemma}
To verify that a cycle $\gamma \in H_5(\P_{\Pi}-(Y_{1,\psi} \cup Y_{2,\psi}))$ lies in the image of $\tube_{1,\psi} \circ \tube_{2,\psi}$, it suffices to find a class $z \in H_6(\P_{\Pi}-\Y_{\psi})$ such that $\delta(z) = \gamma$ and $i_* [z] = 0$, where $i : \P_{\Pi} - \Y_{\psi} \to \P_{\Pi}$ is the canonical open embedding. Here, $\delta$ is the Mayer-Vietoris homomorphism.
\end{lemma}

\begin{proof}
By writing the long exact sequence of relative homology for the pair $(\P_{\Pi},\P_{\Pi} - \Y_{\psi})$, we get:

             $$... \to H_7(\P_{\Pi},\P_{\Pi}-\Y_{\psi}) \to H_6(\P_{\Pi} - \Y_{\psi}) \to H_6(\P_{\Pi}) \to H_6(\P_{\Pi}, \P_{\Pi} - \Y_{\psi}) \to ...$$
    Applying Poincar\'e-Lefschetz and Poincar\'e duality, we find that $H_7(\P_{\Pi},\P_{\Pi}-\Y_{\psi}) \cong H^3(\Y_{\psi}) \cong H_3(\Y_{\psi})$, and under these identifications, the first morphism in the sequence above is isomorphic to the tube map:
    $$H_3(\Y_{\psi}) \xrightarrow{\operatorname{tube}_\psi} H_6(\P_{\Pi} - \Y_{\psi}).$$ 
    It follows that the image of $\tube_\psi$ is the kernel of the map $H_6(\P_{\Pi} - \Y_{\psi}) \to H_6(\P_{\Pi})$.
    By Lemma \ref{tube-lemma}, we know that $\delta \circ \tube_{\psi} = \tube_{1, \psi} \circ  \tube_{2, \psi}$, which completes the proof. 
\end{proof} 
 
\begin{lemma} \label{cycle-lemma}
    For every $\psi \in U$ such that $|\psi|>1$, there exists a cycle $\gamma \in H_3(\Y_{\psi})$ such that 
    $$\operatorname{tube}_{1,\psi} \circ \operatorname{tube}_{2,\psi} (\gamma) = (S^1)^5,$$ where $(S^1)^5 \in H_5(\P_{\Pi}-(Y_{1,\psi} \cup Y_{2,\psi}))$ is a 5-cycle defined by the equations $|t_1|=...=|t_5|=1$ in $\mathbb T_N$.
\end{lemma}
\begin{proof}
    We will write $Y_1,Y_2,\P$ instead of $Y_{1,\psi},Y_{2,\psi},\P_{\Pi}$ throughout the proof. Let us define two cones, which are maximal cones of the fan $\Pi$:
    $$\sigma_1 = \cone(v_{111},v_{112},v_{113},v_{114}, v_{115}) \in \Pi,$$
    $$\sigma_2 = \cone(u_{266},u_{366},u_{466},u_{566},u_{666}) \in \Pi.$$
    The dual cones to the cones above are:
    $$\sigma_1^{\vee} = \cone((-1,-2,-2,-2,-2,0)^T, (0,1,0,0,0,0)^T, (0,0,1,0,0,0)^T,$$ $$(1,1,1,2,1,0)^T, (1,1,1,1,2,0)^T) \subset{M_{\R}},$$
    $$\sigma_2^{\vee} = \cone((-1,1,0,0,0,0)^T, (-1,0,1,0,0,0)^T,(0,0,0,1,0,0)^T $$
    $$(0,0,0,0,1,0)^T,(1,-1,-1,-1,-1,0)^T) \subset M_{\R}.$$
    With the given ordering of the generators of the dual cones, we introduce coordinates:
    $$U_1 := U_{\sigma_1} = \spec \C[M \cap \sigma_1^{\vee}]  = \spec \C[x_1,...,x_5],$$
    $$U_2 := U_{\sigma_2} = \spec \C[M \cap \sigma_2^{\vee}] = \spec \C[y_1,...,y_5].$$
    The ordering of the coordinates corresponds to the given ordering of the generators of the dual cones. For example, $x_3 = t_3, y_1 = \frac{t_2}{t_1}$. Note that since $\sigma_1 \cap \sigma_2 = 0$, we have $U_1 \cap U_2 = \T_N$. We define our cycle as the sum of two parts $z=z_1+z_2$ such that $\supp z_1 \subset U_1 - Y_1 \subset \P - Y_1$ and $\supp z_2 \subset U_2 - Y_2 \subset \P - Y_2$. Hence $\supp(z_1+z_2) \subset \P - \Y_{\psi}$ and $\delta(z_1+z_2) = \partial(z_1)$.
    
    We now define:
    $$z_1 = \{(x_1,...,x_5) \in U_1 : |x_1| \leq 1, |x_2|=|x_3|=|x_4|=|x_5|=1\},$$
    $$z_2 = \{(y_1,...,y_5) \in U_2 : |y_1| \leq 1, |y_2|=|y_3|=|y_4|=|y_5|=1\},$$
    $$z=z_1+z_2.$$
    Notice that $z$ is indeed a cycle, which means that $\partial(z_1+z_2) = [(S^1)^5]-[(S^1)^5]=0$. Notice that we have the identity $\delta(z)=\partial(z_1)=(S^1)^5$, by definition of the Mayer--Vietoris homomorphism. Indeed, the Mayer--Vietoris homomorphism sends the cycle $z=z_1+z_2$ supported on $(U_1 - Y_1) \cup (U_2 - Y_2)$, with $z_1$ supported on the $U_1-Y_1$ and $z_2$ supported on the $U_2 - Y_2$, to the boundary of $z_1$, which is supported on the $(U_1-Y_1) \cap (U_2 - Y_2)$. 
    
    To check that $\supp{z_1} \subset U_1 - Y_1$ we have to compute equations that define the variety $Y_1 \cap U_1$. The computation gives us
    $$3 \psi - x_1^3x_2^2x_3^2x_4^2x_5^2 - x_2 - x_3 = 0,$$
    assuming that $(x_1,...,x_5) \in \supp z_1$ and $(x_1,...,x_5) \in Y_1 \cap U_1$, we can deduce
    $$3\psi = x_1^3x_2^2x_3^2x_4^2x_5^2 + x_2 + x_3,$$
    $$3|\psi| \leq |x_1^3x_2^2x_3^2x_4^2x_5^2| + |x_2| + |x_3| \leq 3,$$ which contradicts the assumption that $|\psi|>1$. Notice that the inequality $|x_i| \leq 1$ follows from the assumption $(x_1,...,x_5) \in \supp z_1$ and definition of $z_1$. Analogously, to prove $\supp(z_2) \subset U_2 - Y_2$, we must write equations that define $Y_2 \cap U_2$. The computation gives us
    $$3 \psi - y_4 - y_5 - y_1^2 y_2^2 y_3^2 y_4^2 y_5^3 = 0,$$
    the same sort of reasoning shows that $\supp(z_2) \subset U_2 - Y_2$. 

    The last thing we have to check is that the class $[z]$ is $0$ as an element of $H_6(\P)$. By Poincar\'e duality, $[z] = 0$ if and only if $\operatorname{PD}([z])\operatorname{PD}([y]) = 0$ for all $[y] \in H_4(\P)$, where $\operatorname{PD} : H_i(\P) \to H^{10-i}(\P)$ is the Poincar\'e duality map. Since the cohomology ring of a smooth complete toric variety is generated by divisor classes, it is enough to show that
    \[
    \operatorname{PD}([z])\cdot [D_{\rho_1}]\cdot [D_{\rho_2}]\cdot [D_{\rho_3}]=0
    \]
    for all \(\rho_1,\rho_2,\rho_3\in\Pi(1)\). We will in fact show the stronger statement
    $$\operatorname{PD}([z])\cdot [D_\rho]=0$$
    for every torus-invariant prime divisor $D_\rho \subset \P$.
    
    Recall that $\operatorname{supp} z\subset U_1\cup U_2$. If $\rho\notin \sigma_1(1)\cup\sigma_2(1),$ then 
    $D_\rho\cap U_1=\emptyset, D_\rho\cap U_2=\emptyset.$ Hence $D_\rho\cap \operatorname{supp} z=\emptyset,$ and therefore $\operatorname{PD}([z])\cdot [D_\rho]=0$. It remains to consider the divisors corresponding to rays of $\sigma_1$ and $\sigma_2$. Let $D_{x_i}$ denote the torus-invariant divisor whose equation on $U_1$ is $x_i=0$, and let $D_{y_i}$ denote the torus-invariant divisor whose equation on $U_2$ is $y_i=0$. 

    For $i=2,\ldots,5$, we have
    $$D_{x_i}\cap \operatorname{supp} z=\emptyset.$$
    Indeed, $D_{x_i}\cap U_2=\emptyset$, since the corresponding ray belongs to $\sigma_1(1)$ and not to $\sigma_2(1)$, while on $U_1$ the support of $z_1$ satisfies $|x_i|=1.$ Thus $x_i\neq 0$ on $\operatorname{supp} z_1$. Hence
    $$ \operatorname{PD}([z])\cdot [D_{x_i}]=0 \qquad \text{for } i=2,\ldots,5.$$
    A similar argument applies to $D_{y_i}$ for $i=2,\ldots,5$.

    Therefore, the only remaining divisors are $D_{x_1}$ and $D_{y_1}$. We first treat $D_{x_1}$. We have $D_{x_1} \cap U_2 = 0$, hence $z \cap D_{x_1} = z_1 \cap D_{x_1}$. In the coordinates on $U_1$, this intersection is
    $$z_1\cap D_{x_1}=\{x_1=0,\ |x_2|=|x_3|=|x_4|=|x_5|=1\}\subset U_1.$$
    The intersection is transverse, because $x_1$ is the disk coordinate in the definition of $z_1$. Hence
    $$\operatorname{PD}([z])\cdot [D_{x_1}] =\operatorname{PD}([z_1\cap D_{x_1}]).$$
    But $z_1\cap D_{x_1}$ is a positive-dimensional cycle supported in $U_1\cong \mathbb C^5$. Since 
    $$H_k(U_1)=0 \qquad \text{for } k>0, $$
    this cycle is null-homologous in $U_1$, hence also null-homologous in $\P$. Therefore $\operatorname{PD}([z])\cdot [D_{x_1}]=0$. The same argument applies to $D_{y_1}$. 

    We have shown that $$\operatorname{PD}([z])\cdot [D_\rho]= \emptyset $$for every torus-invariant prime divisor $D_\rho$. Consequently, $$\operatorname{PD}([z])\cdot [D_{\rho_1}]\cdot [D_{\rho_2}]\cdot [D_{\rho_3}]=0$$
    for all $\rho_1,\rho_2,\rho_3\in\Pi(1)$. Since divisor classes generate $H^*(\P)$, it follows that
    $$\operatorname{PD}([z])=0.$$
    Hence $[z]=0$ as a class in $H_6(\P)$.
\end{proof}

\begin{corollary}\label{cycle-lemma-family}
    For any open set $W \subset U^{an} \cap \{|\psi|>1\}$, there exists a section $\gamma \in \Gamma(W,(R^3 f_* \C)^{\vee})$ such that, $\tube_{1,\psi} \circ \tube_{2,\psi} (\gamma(\psi)) = (S^1)^5$ for each $\psi \in W$.
\end{corollary}

We consider the period along the constructed cycle:
   \begin{equation}\label{eq:phi0}
       \Phi_0 := \int_{\gamma} \omega = \frac{1}{(2 \pi i)^5}\int_{(S^1)^5} \omega^{\circ} = \sum_{n \geq 0} \frac{((3n)!)^2}{(n!)^6} (3 \psi)^{-6n}.
   \end{equation}
The period appears in \cite{BvS}. The first equality follows from the definition of $\omega$, given in \eqref{eq:etacirc-def}. The class $\omega \in H^3(\Y_{\psi})$ is a (double) residue, and the residue is a map dual to the (double) tube map --- hence the first equality in \eqref{eq:phi0}. One can use the geometric expansion
$$\frac{1}{1 - \frac{t_1}{3 \psi}  - \frac{t_2}{3 \psi}  - \frac{t_3}{3 \psi}} \frac{1}{1 - \frac{t_4}{3 \psi}  - \frac{t_5}{3 \psi}  - \frac{t_6}{3 \psi}} = \sum_{k, l \geq 0} (3 \psi)^{-k - l} (t_1 + t_2 + t_3)^k (t_4 + t_5 + t_6)^l.$$
Only the terms $(t_1 ... t_6)^{s} = 1$ (with coefficient $(3 \psi)^{-6s}$) will survive integrating via the Cauchy formula. A combinatorial counting argument shows that the expression $$(t_1 + t_2 + t_3)^{3s} (t_4 + t_5 + t_6)^{3s}$$ 
includes a term 
$$\frac{(3s)! (3s)!}{(s!)^6} (t_1 ... t_6)^s$$ 
in its monomial expansion. This is a good time to introduce the new coordinate 
$$z = (3 \psi)^{-6},$$
since we can see from \eqref{eq:phi0}, that our period depends only on the variable $z$. In the $z$-coordinate, the smooth fiber locus is $V =\P^1 - \{0, \infty, 3^{-6}\}$.

\subsection{Irreducibility of the Picard--Fuchs equation}
Let $U = \P^1 - (\{0,\infty\} \cup \mu_6)$ be the locus of smooth fibers, and let $\mathcal E := \mathcal O_U \otimes_{\C} R^3 f_* \C$ be the algebraic coherent free $\mathcal O_U$-module equipped with an algebraic flat Gauss--Manin connection $\nabla : \mathcal E \to \Omega^1_U \otimes_{\mathcal O_U} \mathcal E$. Let $\mathcal E(\psi)$ be the fiber of $\mathcal E$ over the point $\psi$; sometimes $H^3(\Y_{\psi})$ will be used instead, since the two vector spaces are canonically isomorphic.

We will need the Hodge numbers of $\Y_{\psi}$ for $\psi \in U$. Since the family $\Y_{\psi}$ was constructed as a (crepant desingularization of a) Batyrev--Borisov mirror dual to the intersection of two cubics in $\P^5$. We can use the computations of the Hodge numbers for the intersection of two general cubics in $\P^5$ and \cite[Theorem 4.15]{StringyBB}, in order to obtain the Hodge numbers for $\Y_{\psi}$. The precise statement: we have $h^{p,q}(\Y_{\psi}) = 1$ for every $\psi \in U$ and every $p+q = 3$.

Let $T = T_{\psi_0} : \mathcal E (\psi_0) \to \mathcal E(\psi_0)$ be the local monodromy operator near the point $\infty$, where $\psi_0 \in D_{\infty} - \{\infty\} = \{\psi \in U : |\psi|  > 1\}$ is a fixed distinguished point. To prove that $\infty$ is a point of maximal unipotent monodromy, means that $T$ is unipotent and $\log T$ is a nilpotent operator of the maximal possible degree. In other words, the Jordan normal form of the operator $T$ consists of a single Jordan block with the eigenvalue $1$. Following \cite{Mor}, we will extract information about the operator $T$ from the Picard--Fuchs equation, that we will introduce shortly. 

Consider the following differential operator:
    $$L = D^4 - 3^6 z \left(D+\frac{1}{3}\right)^2\left(D+\frac{2}{3}\right)^2, \quad D = z \frac{d}{dz},$$
or, equivalently, in terms of the coordinate $\psi$,
    $$L = \frac{1}{6^4}\left(Q^4 - \psi^{-6} \left(Q - 2 \right)^2\left(Q - 4\right)^2\right), \quad Q = \psi \frac{d}{d \psi}.$$

The operator has degree 4. In Theorem \ref{th:v-basis}, we will show that it is the minimal-degree operator that annihilates the form $\omega$ --- a property that defines the Picard--Fuchs operator. The associated differential equation $L [\phi] = 0$ is known as the Picard--Fuchs equation. It can be verified directly using the Taylor series  expansion \eqref{eq:phi0} that 
\begin{equation} \label{eq:Lphi0=0}
L[\Phi_0] = 0, 
\end{equation}  
where $\Phi_0$ is the holomorphic period defined in \eqref{eq:phi0}.

Depending on the context, it may be more natural to express the operator in terms of the $z$-coordinate or the $\psi$-coordinate, which is why both expressions are provided. We will show that $L$ annihilates the section $\omega$ and that $L$ can't be decomposed into a product of two differential operators with coefficients in $\C(z)$. Such operators are called irreducible. Thus, $L$ is the irreducible differential operator that annihilates $\omega$. 

Let 
$$A = \C(z) \left[z \frac{d}{dz}\right]$$
be the noncommutative ring of polynomial differential operators with coefficients in $\C(z)$. The ring $A$ is a left Euclidean ring. Recall that a differential operator $L$ in $A$ is irreducible if it cannot be written as a product $P Q$, where $P, Q$ are non-constant polynomial differential operators with coefficients in $\C(z)$.  \cite[Proposition 2.11.9(2), $\alpha = 3^{-6}$]{Katz} states that the operator $L$ is irreducible as an element of $A$. We combine Katz's irreducibility result with the Yukawa coupling computation of \cite[Section 5.6.1]{BvS} to prove the following.`
\begin{theorem} \label{th:v-basis}
   The equation $L \omega = 0$ holds. The classes 
   $$D^i \omega(z), \qquad i = 0, \ldots, 3 $$ 
   form a basis of $\mathcal E (z)$ for each $z \in V$. 
\end{theorem}
\begin{proof}
    By a computation of the Hodge numbers, there exists a $\C(z)$-linear dependence  of the form
    $$a_0 \omega + a_1 D \omega + ... + a_4  D^4 \omega = 0,$$ 
    for some $a_i \in \C(z)$. Hence, there exists an operator 
    $$P = D^m + a_{m-1} D^{m-1} + \ldots + a_0 $$ 
    satisfying $P (\omega) = 0$. Furthermore, we assume that $P$ has minimal possible order among all nonzero differential operators with coefficients in $\C(z)$ that annihilate $\omega$, and that $P$ is monic --- that is, leading coefficient is equal to $1$. Note that $m =\deg P \leq \deg L = 4$. 
  
   Integrating the equation $P(\omega) = 0$ along the cycle $\gamma$, we obtain 
   $$P \left(\int_{\gamma} \omega \right) = 0,$$  
   hence, $P(\Phi_0) = 0$. The period $\Phi_0$, defined in \eqref{eq:phi0}, is holomorphic when $|z| < 3^{-6}$. By the definition of $P$, we find the identity $P(\Phi_0) = 0$. So, we have both $P (\Phi_0) = 0$ and $L (\Phi_0) = 0$ when $z < 3^{-6}$.
    
    We apply the Euclidean algorithm in the left Euclidean ring $A$ to show that $L = q P$ for some $q \in \C(z)$. We write $L = q P + r$ for some $q,r \in A$ with $\deg r < \deg P \leq 4$.  Here, $\deg r$ (or $\deg P$) refers to the order of $r$ (or $P$) as a differential operator. Since $L(\Phi_0)=0$ and $P(\Phi_0)=0$, it follows that $r(\Phi_0)=0$. 
    
    If $r = 0$, then $\deg q = 0$, since $\deg q \neq 0$ would contradict the irreducability of $L$ in $A$, in that case, we have $L = q P$ for some $q \in \C(z)$, which is precisely what we want. 
    
    If $r \neq 0$, we can repeat the argument with the pair $(L, r)$ instead of $(L,P)$, constructing a new decomposition $L = q_1 r + r_1$. We repeat the process until the remainder $r_i$ is $0$. After at most $4$ steps, we conclude that $L = q_{*} r_*$ for some $q_*,r_* \in A$ with $\deg q_* + \deg r_* = 4$. We can rule out the case $\deg r_* = 0$, since it would imply the impossible equality $r_* \Phi_0 = 0$. Thus $1 \leq \deg r_* \leq 3$, which contradicts the irreducability of $L$, as established in \cite[Proposition 2.11.9(2), $\alpha = 3^{-6}$]{Katz}.

    To show that $\omega(z), \ldots, D^3 \omega(z)$ form a basis of $\mathcal E(z)$ for $z \in V$, we can consider a Yukawa coupling. In \cite[Section 5.6.1]{BvS}, it was shown that
    $$Y(z) = \int_{\Y_{z}} \omega \wedge D^3 \omega, \qquad DY = - \frac{1}{2}\left( - \frac{2 \cdot 3^6 z}{1 - 3^6z}\right) Y, \qquad Y = C \frac{1}{1 - 3^6 z},$$
    for some nonzero constant $C$. 

    This implies that the class of $D^3\omega(z)$ in $\F^0\mathcal E(z)/\F^1\mathcal E(z)$ is nonzero. Indeed, since $\omega(z)\in \F^3\mathcal E(z)$ and the pairing with $\omega(z)$ vanishes on $\F^1\mathcal E(z)$. Thus the functional
    $$\alpha\mapsto \int_{\Y_z}\omega(z)\wedge \alpha$$
    factors through the quotient $\F^0\mathcal E(z)/\F^1\mathcal E(z)$. Since
    $$\int_{\Y_z}\omega(z)\wedge D^3\omega(z)=Y(z)\neq 0,$$
    the image of $D^3\omega(z)$ in $\F^0\mathcal E(z)/\F^1\mathcal E(z)$ is nonzero.
    
    The statement that $\omega(z), ..., D^3 \omega(z)$ form a basis of $\mathcal E(z), z \in V$ is equivalent to the condition that $0 \neq [D^3 \omega(z_0)]$ in the quotient $\F^0 \mathcal E(z_0)/\F^1 \mathcal E (z_0)$, where $\F^k \mathcal E$ denotes the Hodge filtration on $\mathcal E$ and $[D^3 \omega]$ represents the corresponding quotient class. Therefore each of the vectors
    $$\omega(z),\quad D\omega(z),\quad D^2\omega(z),\quad D^3\omega(z)$$
    has a nonzero image in the corresponding one-dimensional graded piece
    $$
    \F^3\mathcal E(z),\quad
    \F^2\mathcal E(z)/\F^3\mathcal E(z),\quad
    \F^1\mathcal E(z)/\F^2\mathcal E(z),\quad
    \F^0\mathcal E(z)/\F^1\mathcal E(z).
    $$
    Hence these four vectors are linearly independent. Since $\dim \mathcal E(z)=4$, they form a basis of $\mathcal E(z)$ for every $z\in V$.
\end{proof}

The following corollary shows that the statements corresponding to Theorem \ref{th:v-basis} hold over $U$, i.e., in the variable $\psi$.
\begin{corollary}\label{th:u-basis}
The global section $\omega \in \Gamma(U, \mathcal E)$ satisfies the equation $L \omega = 0$. The elements $D^i \omega(\psi),i =0, \ldots, 3$, form a basis of $\mathcal E(\psi)$ for each $\psi \in U$. Hence, $L$ is the Picard--Fuchs operator associated with $\omega$. 
\end{corollary}

Using the the Picard--Fuchs equation and \cite[Section 1, Corollary]{Mor}, we can deduce.

\begin{proposition}\label{th:infinity-is-mum}
The variation of Hodge structures associated with the local system $R^3 f_* \C$ has maximally unipotent local monodromy near $\psi = \infty$. 
\end{proposition}

\section{0 is a K-point} \label{sec:k-point}
In this section we investigate the special fiber $\Y_0$ from a Hodge-theoretical perspective. For the purposes of this discussion, we define a $K$-point to be a point where the local monodromy is unipotent and has two Jordan blocks of size two. Points in the parameter space corresponding to fibers with this property are sometimes referred to as $K$-points; see, for example, \cite{DoranHarderThompson}. We denote the Hodge--Deligne numbers of the limiting mixed Hodge structure near a point $a$ by 
\begin{equation}\label{def:hpqlim} h^{p,q} = h^{p,q}(H^3_{\lim,a}) = \dim_{\C} Gr_F^{p} Gr^W_{p+q} 
H^3_{\lim,a}.
\end{equation}
See \cite{Schmid} for a precise definition of limiting mixed Hodge structures. Analogously to the usual Hodge diamond, we consider the Hodge--Deligne diamond of $H^3_{\lim,0}$, with the numbers $h^{p,q}$ defined as in \eqref{def:hpqlim}:
$$
\begin{matrix}
  &   &    & h^{3,3} &    &   &   \\
  &   & h^{3,2}  &   & h^{2,3}  &   &   \\
  &  h^{3,1} &    & h^{2,2} &    & h^{1,3} &   \\
 h^{3,0} &   &  h^{2,1} &   &  h^{1,2} &   &  h^{0,3} \\
  &  h^{2,0} &    &  h^{1,1} &    &  h^{0,2} &   \\
  &   &  h^{1,0} &   &  h^{0,1}  &   &   \\
  &   &    & h^{0,0} &    &   &  
\end{matrix}
$$
The diamond is symmetric with respect to both the vertical and horizontal axes. In this section, we will prove that the Hodge--Deligne diamond of the limiting mixed Hodge structure near the $0$ point is as follows:
$$
\begin{matrix}
  &   &    & 0 &    &   &   \\
  &   & 0  &   & 0  &   &   \\
  & 1 &    & 0 &    & 1 &   \\
0 &   & 0 &   & 0 &   & 0 \\
  & 1 &    & 0 &    & 1 &   \\
  &   & 0  &   & 0  &   &   \\
  &   &    & 0 &    &   &  
\end{matrix}
$$
(see Theorem \ref{RhombusThoerem}). We also remind the reader that the differential operator 
$$L =  \left(-\frac{1}{6} Q\right)^4 - \psi^{-6} \left(-\frac{1}{6} Q + \frac{1}{3}\right)^2\left(-\frac{1}{6} Q+\frac{2}{3}\right)^2,\qquad Q = \psi \frac{d}{d \psi},$$
annihilates the section $\omega(\psi)$ for every $\psi \in U$, where $U = \P^1 - (\{0,\infty\} \cup \mu_6)$ is the smooth fiber locus of the family $\Y_{\psi}$. The section $\omega$ is defined by the formula \eqref{eq:omega-def}. Thus, the operator $\psi^6 Q^4 -  (Q - 2)^2 (Q - 4)^2$ annihilates the section $\omega(\psi), \psi \in U$, and the operator $\psi^6 (Q+2)^4 - Q^2 (Q-2)^2$ annihilates the section $\psi^{-2} \omega(\psi)$, for all $\psi \in U$. Let us define
\begin{equation} \label{eq:r-def}
    R := \frac{\psi^6}{\psi^6- 1} (Q+2)^4 - \frac{1}{\psi^6- 1}Q^2 (Q-2)^2.
\end{equation}
The normalization is chosen in such a way to ensure that the leading coefficient is equal to $1$. The next proposition is the key step in computing the Hodge--Deligne diamonds.

\begin{proposition} \label{th:solutions-Picard--Fuchs} 
    A fundamental system of solutions to the equation 
    $$R[\phi] = 0$$ 
    consists of four linearly independent, multivalued solutions:
    $$a_0, b_0 + 
   \log(\psi) a_0, a_2, b_2 + \log(\psi) a_2$$ 
    where $a_0,a_2,b_0,b_2$ are holomorphic in a neighborhood of $0$.
\end{proposition}
\begin{proof}
    After algebraic expansion of the expression \eqref{eq:r-def}, we obtain
    $$Q^4 + \frac{8\psi^6+4}{\psi^6-1} Q^3 + \frac{24\psi^6-4}{\psi^6-1} Q^2 + \frac{32\psi^6}{\psi^6-1} Q + \frac{16\psi^6}{\psi^6-1} = $$ 
    $$Q^2 (Q-2)^2 - \psi^6 (12 Q^3 + 20Q^2 + 32Q + 16) -  \psi^{12} (12 Q^3 + 20Q^2 + 32Q + 16) + \dots$$
    We will use the Frobenius method. For a detailed exposition of the method, see \cite[Chapter 4, Section 8]{CL}. The indicial equation in this case is $d(\lambda) = \lambda^2(\lambda-2)^2$.  Following the method, we construct a formal solution of the form $\phi_{\lambda} (\psi) = \psi^{\lambda} + c_1 \psi^{\lambda+1} + \dots$ such that $R[\phi_{\lambda}(\psi)] = d(\lambda) \psi^{\lambda}$. Applying the operator $R$ term-by-term to the formal series $\phi_{\lambda}(\psi)$, we obtain the formal series
    $$R[\phi_{\lambda}(\psi)] = \psi^{\lambda} d(\lambda) + \psi^{\lambda+6} (d(\lambda+6) c_6 -  (12\lambda^3 + 20 \lambda^2 + 32 \lambda + 16)) + \dots$$
    The coefficients of the solution can be constructed recursively, so that $$\phi_{\lambda} = \psi^{\lambda} + \frac{(12\lambda^3 + 20 \lambda^2 + 32 \lambda + 16) }{(\lambda+6)^2 (\lambda + 4)^2} \psi^{\lambda+6} + \dots$$
    It is a formal series with coefficients in $\C(\lambda)$ that satisfies $R[\phi_{\lambda}] = d(\lambda) \psi^{\lambda}$. The zeros of the denominators of the coefficients lie in the set 
    $$\{\dots,-18,-16,-12,-10,-6,-4\}.$$ 
    Hence $\phi_0$ and $\phi_2$ are well-defined solutions. Two additional solutions are obtained by differentiating $\phi_{\lambda}$ with respect to $\lambda$ and evaluating at $\lambda = 0$ and $\lambda = 2$. Writing
    $$
    a_0=\phi_0,\quad a_2=\phi_2, \quad
    b_0+ (\log \psi)a_0=\left.\frac{\partial\phi_\lambda}{\partial\lambda}\right|_{\lambda=0}
    ,\quad b_2 + (\log \psi) a_2=\left.\frac{\partial\phi_\lambda}{\partial\lambda}\right|_{\lambda=2}
    $$
    we obtain the claimed form of the solutions. The solutions are linearly independent since their leading terms are, respectively $1, \log \psi, \psi^2, \psi^2 \log \psi$. 
\end{proof}

We define the normalized form
$$\tilde \omega=(3\psi)^{-2}\omega.$$
Period pairing with $\tilde \omega$ defines a morphism of local systems
$$
(R^3 f_* \mathbb C)^\vee\to \ker R,
\qquad
\gamma\mapsto (\gamma,\widetilde\omega).
$$
Here, by $\ker R$ we mean local system of multivalued solutions investigated in Proposition \ref{th:solutions-Picard--Fuchs}. Since
$$
\tilde \omega,\quad
D\tilde \omega,\quad
D^2\tilde \omega,\quad
D^3\tilde \omega
$$
form a basis of $\mathcal E := R^3f_*\mathbb C\otimes_{\mathbb C}\mathcal O_U$, by Corollary \ref{th:u-basis}, this morphism is injective. As both local systems have rank four, it is an isomorphism. We therefore identify $\ker R$ with $(R^3 f_*\mathbb C)^\vee$. In particular we have the following corollary of Proposition \ref{th:solutions-Picard--Fuchs}.
\begin{corollary} \label{cor:unipotency-of-0}
    The local system $\ker R$ on the punctured disk $0 < |\psi| < 1$ has unipotent local monodromy. Hence, the local monodromy of the local system $R^3 f_* \C$ near $0$ is unipotent as well. For local monodromy operator $T$ we, moreover, have $(T-I)^2 = 0$.
\end{corollary}
\begin{proof}
By Proposition \ref{th:solutions-Picard--Fuchs}, every local solution of the equation
$$R[\phi]=0$$
is a linear combination of
$$
a_0,\quad
b_0+a_0\log\psi,\quad
a_2,\quad
b_2+a_2\log\psi,
$$
where $a_0,a_2,b_0,b_2$ are holomorphic near $\psi=0$. Under analytic continuation around $\psi=0$, one has
$$
\log\psi\mapsto\log\psi+2\pi i,
$$
while the holomorphic functions \(a_0,a_2,b_0,b_2\) remain unchanged. Hence
$$
(T-I)a_0=(T-I)a_2=0,
$$
and
$$
(T-I)(b_0+a_0\log\psi)=2\pi i\,a_0,\qquad
(T-I)(b_2+a_2\log\psi)=2\pi i\,a_2.
$$
Applying \(T-I\) once more gives
$$
(T-I)^2=0.
$$
In particular, the monodromy operator on \(\ker R\) is unipotent. Since
$$
\ker R \cong (R^3f_*\mathbb C)^\vee,
$$
the dual monodromy representation is unipotent, and therefore so is the monodromy representation on $R^3f_*\mathbb C$.
\end{proof}

The vector bundle $\mathcal E = R^3 f_* \C \otimes_{\C} \O_U$ has a Hodge filtration by vector subbundles:
$$\F^3 \mathcal E \subset \F^2 \mathcal E \subset \F^1 \mathcal E \subset \F^0 \mathcal E = \mathcal E,$$ 
which comes from the Hodge filtration of cohomology of the fibers. Let $D = \{|\psi| < 1\} \subset \P^1$ be the analytic disk centered at $0$. Since the local monodromy is unipotent, the Deligne extension of $\mathcal F^3 \mathcal E$ is $f_\ast K_{\Y/D}$ by \cite[Theorem 2.6]{Kollar2}. The fiber at $0$ of this sheaf is nothing but $H^0(\Y_0, K_{\Y_0})$. This is likewise the definition of $F^3 H^3_{\lim,0}$ in \cite{Ste}. This description is useful for the following statement:

\begin{lemma} \label{lemma:omega0-at-zero-is-not-zero}
    The class $(3\psi)^{-2} \omega$ extends as a holomorphic section over the Deligne extension of $\mathcal E$, and $((3\psi)^{-2} \omega) (0) \neq 0$ as an element of $H^3_{\lim,0}$. 
\end{lemma}
\begin{proof}
    Recall that the factor $(3\psi)^2$ was introduced merely to ensure that $\omega$ behaves well near $\infty$. It then follows from the construction and the discussion surrounding Lemma \ref{lemma:omegaconstruction} that $(3\psi)^{-2} \omega$ defines a section of $f_\ast K_{\Y/D},$ which proves the first part of the lemma. 
    Now, from the discussion preceding the lemma, it suffices to show that the class is nonzero in $H^0(\Y_0, K_{\Y_0})$. For this, it is enough to verify that its restriction to $K_{\Y_0^{sm}}$ is nonzero, where $\Y_0^{sm} \subset \Y_0$ is the nonsingular locus. 
    
    Over this locus, the sheaf $K_{\Y_0^{sm}}$ can be described in terms of holomorphic top forms via the adjunction formula realized through Poincar\'e residues. Since this locus is contained in the open torus, we can compute $(3 \psi)^{-2} \omega |_{\psi = 0}$ in toric coordinates using the formula \eqref{eq:etacirc-def}. For simplicity of notation, we write $Y_1, Y_2$ instead of $Y_{1, 0}, Y_{2, 0}$ and we assume that $\psi = 0$ throughout the computation except for the initial multiplication by $(3 \psi)^{-2}$. We have:
    $$(3 \psi)^{-2} \omega^{\circ}|_{\T_N - (Y_{1} \cup Y_{2})} = \frac{1}{(2\pi i)^3} \frac{1}{(t_1 + t_2 + t_3) (t_4 + t_5 + t_6)} \frac{dt_1}{t_1} \wedge ... \wedge \frac{dt_5}{t_5}.$$
    Its iterated Poincar\'e residue represents the restriction of
    $(3\psi)^{-2}\omega|_{\psi=0}$ to $\Y_0^{sm}\cap\mathbb T_N.$ Let $p \in \Y_0^{\mathrm{sm}}\cap\mathbb T_N$. Since $p$ is a smooth point of the complete intersection there exist $i<j$ such that

    $$\operatorname{Jac}(\th_1,\th_2;t_i,t_j)(p) \neq 0.$$
    Here $\th_1,\th_2$ are defined by the equation \eqref{def:th}. On the corresponding local chart, the residue is, up to sign,
    $$ \frac{1}{(2\pi i)^3} \frac{
    \bigwedge_{k\neq i,j}dt_k
    }{
    t_1t_2t_3t_4t_5\,
    \operatorname{Jac}(\th_1,\th_2;t_i,t_j)
    }.$$
    
    All $t_k$ are nonzero on the torus, and the chosen Jacobian minor is nonzero near $p$. Hence this is a nonzero holomorphic $3$-form near p. Therefore the specialization of $(3\psi)^{-2}\omega$ is nonzero in
    $H^0(\Y_0,K_{\Y_0}),$
    and consequently its value at 0 is nonzero in $H^3_{\lim,0}$.
\end{proof}

It will also be useful to state an analogous result for $\psi = \infty$, which will be needed in subsequent developments.

\begin{lemma}
    $\omega(\infty) \neq 0$ as an element of $H^3_{\lim,\infty}$.
\end{lemma}
\begin{proof}
    The proof is analogous to that of the previous lemma and is omitted.
\end{proof}

Let $T$ be a local monodromy operator near $0$ of the local system $(R^3 f_* \C)^{\vee}$, and $N = \log T$ its logarithm. The following theorem shows that the $0$ point is a $K$-point.
\begin{theorem}\label{RhombusThoerem}
The Hodge--Deligne diamond of the limit mixed Hodge structure near the point $\psi = 0$ for the family $\Y_{\psi}$ is given by:
$$
\begin{matrix}
  &   &    & 0 &    &   &   \\
  &   & 0  &   & 0  &   &   \\
  & 1 &    & 0 &    & 1 &   \\
0 &   & 0 &    & 0  &   & 0 \\
  & 1 &    & 0 &    & 1 &   \\
  &   & 0  &   & 0  &   &   \\
  &   &    & 0 &    &   &  \\
\end{matrix}
$$
\end{theorem}
\begin{proof}
   From the Corollary \ref{cor:unipotency-of-0}, we have $N^2 = 0$, since fundamental system of solutions $R[f] = 0$ has no $(\log \psi)^2$-terms. It follows that $T = 1 + N$. Also we know that the operator $N^k: Gr_{3+k}^{W} H^3_{\lim,0} \to Gr^W_{3-k} H^3_{\lim,0}$ is an isomorphism \cite[Lemma 6.4]{Schmid}. Hence, the rows $0$, $1$, $5$, and $6$ of the Hodge--Deligne diamond are zero. Furthermore, according to \cite[Lemma 6.4]{Schmid}, we know that 
   $$N : Gr^W_4 H^3_{\lim,0} \to Gr^W_2 H^3_{\lim,0}$$ 
   is an isomorphism. Additionally, from explicit computation of Hodge numbers, we know that 
   $$\dim Gr^i_{F} H^3_{\lim,0} = 1$$ 
   for $i=0,1,2,3$. For combinatorial reasons, only three possibilities remain: 
 \[
\begin{matrix}
  &   &    & 0 &    &   &   \\
  &   & 0  &   & 0  &   &   \\
  & 1 &    & 0 &    & 1 &   \\
0 &   & 0 &   & 0 &   & 0 \\
  & 1 &    & 0 &    & 1 &   \\
  &   & 0  &   & 0  &   &   \\
  &   &    & 0 &    &   &
\end{matrix}
\qquad
\begin{matrix}
  &   &    & 0 &    &   &   \\
  &   & 0  &   & 0  &   &   \\
  & 0 &    & 1 &    & 0 &   \\
1 &   & 0 &   & 0 &   & 1 \\
  & 0 &    & 1 &    & 0 &   \\
  &   & 0  &   & 0  &   &   \\
  &   &    & 0 &    &   &
\end{matrix}
\qquad
\begin{matrix}
  &   &    & 0 &    &   &   \\
  &   & 0  &   & 0  &   &   \\
  & 0 &    & 0 &    & 0 &   \\
1 &   & 1 &   & 1 &   & 1 \\
  & 0 &    & 0 &    & 0 &   \\
  &   & 0  &   & 0  &   &   \\
  &   &    & 0 &    &   &
\end{matrix}
\]
    Through the period isomorphism
    $$(R^3f_*\C)^\vee \cong \ker R,$$
    choose a basis of multivalued flat sections $\gamma_0,\gamma_1,\gamma_2,\gamma_3$ whose periods are the Frobenius solutions:
    $$(\gamma_0,\widetilde\omega)=a_0,\quad
    (\gamma_1,\widetilde\omega)=b_0+a_0\log\psi,\quad
    (\gamma_2,\widetilde\omega)=a_2,\quad
    (\gamma_3,\widetilde\omega)=b_2+a_2\log\psi,$$
    here $a_0,a_2,b_0,b_2$ are the same as in Proposition \ref{th:solutions-Picard--Fuchs}. Then we have
    $$w_i := (\gamma_i,  \tilde \omega) = \left(\gamma_i - \frac{\log \psi}{2 \pi i} N \gamma_i, \tilde \omega \right) + \frac{\log \psi}{2 \pi i} (N \gamma_i, \tilde \omega).$$ 
    The functions $(\gamma_i - \frac{\log \psi}{2 \pi i} N \gamma_i, \tilde \omega)$ and $(N \gamma_i, \tilde \omega)$ are both holomorphic at $0$, since they are pairings of holomorphic objects that extend over $0$. Indeed, by Lemma \ref{lemma:omega0-at-zero-is-not-zero} $\tilde \omega$ is naturally defined at $0$; the cycle $\gamma_i - \frac{\log \psi}{2 \pi i} N \gamma_i$ is invariant under local monodromy, since $N^2 = 0$; and $N \gamma$ is also invariant under local monodromy, again because $N^2 = 0$. Thus, the solutions $w_i$, can be listed as:
    $$w_0 = a_0,\quad w_1 = b_0 +  a_0 \log \psi,\quad w_2 = a_2,\quad w_3 = b_2 + a_2 \log \psi,$$
    where all functions $a_0, b_0, a_2, b_2$ are holomorphic in a neighborhood of $\psi = 0$. Since $a_0 \neq 0, a_2 \neq 0$, it follows that $N \gamma_1 \neq 0, N \gamma_3 \neq 0$ for $i=1$ and $i=3$. Given that $w_0,...,w_3$ form a fundamental system of solutions of the equation $R[\phi] = 0$, as stated in Proposition \ref{th:solutions-Picard--Fuchs}, we conclude that there are two linearly independent solutions whose logarithmic parts are nonzero. This implies that there are two linearly independent cycles --- specifically $\gamma_1$ and $\gamma_3$, such that $N\gamma_1 \neq 0$ and $N\gamma_3 \neq 0$. Moreover $N \gamma_1, N \gamma_3$ are linearly independent since their periods are proportional to $a_0$ and $a_2$ respectively. 
    
    Since $N \gamma_1, N \gamma_3$ are linearly independent, $N$ has rank $2$; hence $N$ has two Jordan blocks of size $2$. By the preceding discussion, $\operatorname{rank} N = 2$, so only the first of the three Hodge--Deligne diamonds can occur. 
\end{proof}
    
    The following proposition shows that the only potentially nonconstant local system $R^if_* \C$ is when $i = 3$. Note that for $i=1,5$, the local system is trivial and for $i=0,6$ the local system is constant.

\begin{proposition}
    The local systems $R^2 f_* \C$ and $R^4 f_* \C$ are constant.
\end{proposition}
\begin{proof}
By \cite[Proposition 1.4]{Mavl}, the restriction map
\[
H^2(\P_{\Pi}) \longrightarrow H^2(\Y_{\psi})
\]
is surjective provided that the polynomials $h_1$ and $h_2$, defined in \eqref{eq:h1} and \eqref{eq:h2}, belong to the irrelevant ideal $B(\Pi)$. Equivalently, every monomial of $h_1$ and $h_2$ must be divisible by a noncone monomial, which is true.

The morphism above assembles into a morphism of local systems. Since $H^2(\P_{\Pi})$ is constant as a local system, its monodromy is trivial. The above surjectivity therefore implies that the local system $R^2 f_* \C$ also has trivial monodromy and is consequently constant. Finally, Poincaré duality identifies the monodromy representations on $R^2 f_* \C$ and $R^4 f_* \C$, so $R^4 f_* \C$ is constant as well.
\end{proof}

\bibliographystyle{alpha}
\bibliography{Biblio}

@incollection {BB,
    AUTHOR = {Batyrev, V. and Borisov, L.},
     TITLE = {On {C}alabi-{Y}au complete intersections in toric varieties},
 BOOKTITLE = {Higher-dimensional complex varieties ({T}rento, 1994)},
     PAGES = {39--65},
 PUBLISHER = {de Gruyter, Berlin},
      YEAR = {1996},
   MRCLASS = {14J32 (14M25)},
  MRNUMBER = {1463173},
MRREVIEWER = {Mark Gross},
}

@article {Esd,
    AUTHOR = {Edelsbrunner, H. and Grayson, D. R.},
     TITLE = {Edgewise subdivision of a simplex},
      NOTE = {ACM Symposium on Computational Geometry (Miami, FL, 1999)},
   JOURNAL = {Discrete Comput. Geom.},
  FJOURNAL = {Discrete \& Computational Geometry. An International Journal
              of Mathematics and Computer Science},
    VOLUME = {24},
      YEAR = {2000},
    NUMBER = {4},
     PAGES = {707--719},
      ISSN = {0179-5376},
   MRCLASS = {52B70},
  MRNUMBER = {1799608},
       DOI = {10.1145/304893.304897},
       URL = {https://doi.org/10.1145/304893.304897},
}

@article {StringyBB,
    AUTHOR = {Batyrev, V. and Borisov, L.},
     TITLE = {Mirror duality and string-theoretic {H}odge numbers},
   JOURNAL = {Invent. Math.},
  FJOURNAL = {Inventiones Mathematicae},
    VOLUME = {126},
      YEAR = {1996},
    NUMBER = {1},
     PAGES = {183--203},
      ISSN = {0020-9910},
   MRCLASS = {14J32 (14M25 32J18 81T30)},
  MRNUMBER = {1408560},
MRREVIEWER = {Anatoly Libgober},
       DOI = {10.1007/s002220050093},
       URL = {https://doi.org/10.1007/s002220050093},
}

@book {Katz,
    AUTHOR = {Katz, N.},
     TITLE = {Exponential sums and differential equations},
    SERIES = {Annals of Mathematics Studies},
    VOLUME = {124},
 PUBLISHER = {Princeton University Press, Princeton, NJ},
      YEAR = {1990},
     PAGES = {xii+430},
      ISBN = {0-691-08598-6; 0-691-08599-4},
   MRCLASS = {14D10 (11L03 11T23 14G15)},
  MRNUMBER = {1081536},
MRREVIEWER = {Hernando Enrique Sierra-Morales},
       DOI = {10.1515/9781400882434},
       URL = {https://doi.org/10.1515/9781400882434},
}

@incollection {Mor,
    AUTHOR = {Morrison, D.},
     TITLE = {Picard-{F}uchs equations and mirror maps for hypersurfaces},
 BOOKTITLE = {Essays on mirror manifolds},
     PAGES = {241--264},
 PUBLISHER = {Int. Press, Hong Kong},
      YEAR = {1992},
   MRCLASS = {32G20 (14D05 14J10 14J30 14N10 32J17)},
  MRNUMBER = {1191426},
MRREVIEWER = {Bruce Hunt},
}

@book {CLS,
    AUTHOR = {Cox, D. A. and Little, J. and Schenck, H.},
     TITLE = {Toric varieties},
    SERIES = {Graduate Studies in Mathematics},
    VOLUME = {124},
 PUBLISHER = {American Mathematical Society, Providence, RI},
      YEAR = {2011},
     PAGES = {xxiv+841},
      ISBN = {978-0-8218-4819-7},
   MRCLASS = {14M25 (05A15 05E45 52B12)},
  MRNUMBER = {2810322},
MRREVIEWER = {Ivan Arzhantsev},
       DOI = {10.1090/gsm/124},
       URL = {https://doi.org/10.1090/gsm/124},
}

@article {EFiMM,
    AUTHOR = {Eriksson, D. and Freixas i Montplet, G. and
              Mourougane, C.},
     TITLE = {Singularities of metrics on {H}odge bundles and their
              topological invariants},
   JOURNAL = {Algebr. Geom.},
  FJOURNAL = {Algebraic Geometry},
    VOLUME = {5},
      YEAR = {2018},
    NUMBER = {6},
     PAGES = {742--775}
}

@article {EFiMM2,
    AUTHOR = {Eriksson, D. and Freixas i Montplet, G. and
              Mourougane, C.},
     TITLE = {B{COV} invariants of {C}alabi-{Y}au manifolds and
              degenerations of {H}odge structures},
   JOURNAL = {Duke Math. J.},
  FJOURNAL = {Duke Mathematical Journal},
    VOLUME = {170},
      YEAR = {2021},
    NUMBER = {3},
     PAGES = {379--454}
}

@article {EFiMM3,
    AUTHOR = {Eriksson, D. and Freixas i Montplet, G. and
              Mourougane, C.},
     TITLE = {On genus one mirror symmetry in higher dimensions and the
              {BCOV} conjectures},
   JOURNAL = {Forum Math. Pi},
  FJOURNAL = {Forum of Mathematics. Pi},
    VOLUME = {10},
      YEAR = {2022},
     PAGES = {Paper No. e19, 53},
      ISSN = {2050-5086},
   MRCLASS = {14J33 (14J32 32Q25 58J52)},
  MRNUMBER = {4475251},
MRREVIEWER = {Hsian-Hua\ Tseng},
       DOI = {10.1017/fmp.2022.13},
       URL = {https://doi.org/10.1017/fmp.2022.13},
}

@article {BCOV,
    AUTHOR = {Bershadsky, M. and Cecotti, S. and Ooguri, H. and Vafa, C.},
     TITLE = {Kodaira-{S}pencer theory of gravity and exact results for
              quantum string amplitudes},
   JOURNAL = {Comm. Math. Phys.},
  FJOURNAL = {Communications in Mathematical Physics},
    VOLUME = {165},
      YEAR = {1994},
    NUMBER = {2},
     PAGES = {311--427},
      ISSN = {0010-3616,1432-0916},
   MRCLASS = {32G81 (14J32 14N10 32J81 58G26 81T30 81T40)},
  MRNUMBER = {1301851},
MRREVIEWER = {Bruce\ Hunt},
       URL = {http://projecteuclid.org/euclid.cmp/1104271134},
}

@article {B,
    AUTHOR = {Batyrev, V.},
     TITLE = {Dual polyhedra and mirror symmetry for {C}alabi-{Y}au
              hypersurfaces in toric varieties},
   JOURNAL = {J. Algebraic Geom.},
  FJOURNAL = {Journal of Algebraic Geometry},
    VOLUME = {3},
      YEAR = {1994},
    NUMBER = {3},
     PAGES = {493--535},
      ISSN = {1056-3911},
   MRCLASS = {14J32 (14M25)},
  MRNUMBER = {1269718},
MRREVIEWER = {I. Dolgachev},
}

@article {Schmid,
    AUTHOR = {Schmid, W.},
     TITLE = {Variation of {H}odge structure: the singularities of the
              period mapping},
   JOURNAL = {Invent. Math.},
  FJOURNAL = {Inventiones Mathematicae},
    VOLUME = {22},
      YEAR = {1973},
     PAGES = {211--319},
      ISSN = {0020-9910},
   MRCLASS = {14D05 (14D20 22E15 32J25 32M10)},
  MRNUMBER = {382272},
MRREVIEWER = {Andrew J. Sommese},
       DOI = {10.1007/BF01389674},
       URL = {https://doi.org/10.1007/BF01389674},
}

@book {KunMum,
    AUTHOR = {Kempf, G. and Knudsen, F. and Mumford, D. and
              Saint-Donat, B.},
     TITLE = {Toroidal embeddings. {I}},
    SERIES = {Lecture Notes in Mathematics, Vol. 339},
 PUBLISHER = {Springer-Verlag, Berlin-New York},
      YEAR = {1973},
     PAGES = {viii+209},
   MRCLASS = {14E15 (14D20 14E05 14M20 20G15)},
  MRNUMBER = {0335518},
MRREVIEWER = {G. Harder},
}

@article {BvS,
    AUTHOR = {Batyrev, V. and van Straten, D.},
     TITLE = {Generalized hypergeometric functions and rational curves on
              {C}alabi-{Y}au complete intersections in toric varieties},
   JOURNAL = {Comm. Math. Phys.},
  FJOURNAL = {Communications in Mathematical Physics},
    VOLUME = {168},
      YEAR = {1995},
    NUMBER = {3},
     PAGES = {493--533},
      ISSN = {0010-3616},
   MRCLASS = {32G20 (14J32 14N10 32G81 32J18 33C70)},
  MRNUMBER = {1328251},
MRREVIEWER = {Anatoly Libgober},
       URL = {http://projecteuclid.org/euclid.cmp/1104272487},
}

@book {CL,
    AUTHOR = {Coddington, E. and Levinson, N.},
     TITLE = {Theory of ordinary differential equations},
 PUBLISHER = {McGraw-Hill Book Co., Inc., New York-Toronto-London},
      YEAR = {1955},
     PAGES = {xii+429},
   MRCLASS = {36.0X},
  MRNUMBER = {69338},
MRREVIEWER = {M.\ Zl\'{a}mal},
}

@incollection {Ste,
    AUTHOR = {Steenbrink, J. H. M.},
     TITLE = {Mixed {H}odge structure on the vanishing cohomology},
 BOOKTITLE = {Real and complex singularities ({P}roc. {N}inth {N}ordic
              {S}ummer {S}chool/{NAVF} {S}ympos. {M}ath., {O}slo, 1976)},
     PAGES = {525--563},
 PUBLISHER = {Sijthoff \& Noordhoff, Alphen aan den Rijn},
      YEAR = {1977},
      ISBN = {90-286-0097-3},
   MRCLASS = {14C30 (14B05 14D05)},
  MRNUMBER = {485870},
MRREVIEWER = {Helmut\ Hamm},
}

@article {S,
    AUTHOR = {Stapledon, A.},
     TITLE = {Inequalities and {E}hrhart {$\delta$}-vectors},
   JOURNAL = {Trans. Amer. Math. Soc.},
  FJOURNAL = {Transactions of the American Mathematical Society},
    VOLUME = {361},
      YEAR = {2009},
    NUMBER = {10},
     PAGES = {5615--5626},
      ISSN = {0002-9947},
   MRCLASS = {52B20},
  MRNUMBER = {2515826},
MRREVIEWER = {Benjamin Braun},
       DOI = {10.1090/S0002-9947-09-04776-X},
       URL = {https://doi.org/10.1090/S0002-9947-09-04776-X},
}

@article {Kollar2,
    AUTHOR = {Koll\'ar, J\'anos},
     TITLE = {Higher direct images of dualizing sheaves. {II}},
   JOURNAL = {Ann. of Math. (2)},
  FJOURNAL = {Annals of Mathematics. Second Series},
    VOLUME = {124},
      YEAR = {1986},
    NUMBER = {1},
     PAGES = {171--202}
}

@article {DoranHarderThompson,
    AUTHOR = {Doran, Charles F. and Harder, Andrew and Thompson, Alan},
     TITLE = {Hodge numbers from {P}icard-{F}uchs equations},
   JOURNAL = {SIGMA Symmetry Integrability Geom. Methods Appl.},
  FJOURNAL = {SIGMA. Symmetry, Integrability and Geometry. Methods and
              Applications},
    VOLUME = {13},
      YEAR = {2017},
     PAGES = {Paper No. 045, 23},
      ISSN = {1815-0659},
   MRCLASS = {14D07 (14D05 14J32 32G20)},
  MRNUMBER = {3663590},
MRREVIEWER = {YiHu\ Yang},
       DOI = {10.3842/SIGMA.2017.045},
       URL = {https://doi.org/10.3842/SIGMA.2017.045},
}

@book {Triangulations,
    AUTHOR = {De Loera, Jes\'us A. and Rambau, J\"org and Santos, Francisco},
     TITLE = {Triangulations},
    SERIES = {Algorithms and Computation in Mathematics},
    VOLUME = {25},
      NOTE = {Structures for algorithms and applications},
 PUBLISHER = {Springer-Verlag, Berlin},
      YEAR = {2010},
     PAGES = {xiv+535},
      ISBN = {978-3-642-12970-4},
   MRCLASS = {52B55 (05C10 52B05 57Q15 68U05)},
  MRNUMBER = {2743368},
       DOI = {10.1007/978-3-642-12971-1},
       URL = {https://doi-org.libproxy.ibs.re.kr/10.1007/978-3-642-12971-1},
}

@article {LamPostnikovI,
    AUTHOR = {Lam, Thomas and Postnikov, Alexander},
     TITLE = {Alcoved polytopes. {I}},
   JOURNAL = {Discrete Comput. Geom.},
  FJOURNAL = {Discrete \& Computational Geometry. An International Journal
              of Mathematics and Computer Science},
    VOLUME = {38},
      YEAR = {2007},
    NUMBER = {3},
     PAGES = {453--478},
      ISSN = {0179-5376,1432-0444},
   MRCLASS = {52B12 (52B40)},
  MRNUMBER = {2352704},
MRREVIEWER = {Jes\'us\ A.\ De Loera},
       DOI = {10.1007/s00454-006-1294-3},
       URL = {https://doi-org.libproxy.ibs.re.kr/10.1007/s00454-006-1294-3},
}

@article {Mavl,
    AUTHOR = {Mavlyutov, Anvar R.},
     TITLE = {Cohomology of complete intersections in toric varieties},
   JOURNAL = {Pacific J. Math.},
  FJOURNAL = {Pacific Journal of Mathematics},
    VOLUME = {191},
      YEAR = {1999},
    NUMBER = {1},
     PAGES = {133--144},
      ISSN = {0030-8730,1945-5844},
   MRCLASS = {14M25 (14D07 32S35)},
  MRNUMBER = {1725467},
MRREVIEWER = {T.\ Oda},
       DOI = {10.2140/pjm.1999.191.133},
       URL = {https://doi-org.libproxy.ibs.re.kr/10.2140/pjm.1999.191.133},
}
\end{document}